\documentclass[10pt, a4paper]{amsart}
\usepackage{amsmath, amsfonts,amsthm,amssymb,amscd, verbatim, graphicx,color,multirow,booktabs,tikz,adjustbox,setspace}
\usepackage{chngcntr}
\def\classification#1{\def\@class{#1}}
\classification{\null}
\usepackage{calc}
\usetikzlibrary{matrix, calc, arrows, graphs}
\usepackage{chngcntr, arydshln, soul, float,subcaption}
\def\classification#1{\def\@class{#1}}
\classification{\null}
\usepackage[margin=0.75in]{geometry}
\usepackage[normalem]{ulem}
\newcommand{\Out}{\mathop{\mathrm{Out}}}

\usepackage{mathtools}

\usepackage{hyperref}

\DeclareMathOperator{\aut}{Aut}
\newcommand{\Alt}{\mathop{\mathrm{Alt}}}
\newcommand{\Sym}{\mathop{\mathrm{Sym}}}

\newcommand{\Aut}{\mathop{\mathrm{Aut}}}

\def\cent#1#2{{C}_{{#1}}{{(#2)}}}

\def\spanq#1{\left\langle #1 \right\rangle_{\mathbb{F}_q}}
\def\M#1{{\mathrm{M}}_{{#1}}{{(q)}}}
\newcommand{\End}[1]{\mathrm{End}_{\Fq}{(#1)}}
\newcommand{\Ends}[1]{\mathrm{End}^s_{\Fq}{(#1)}}
\newcommand{\Hom}[2]{\mathrm{Hom}_{\Fq}{(#1,#2)}}
\newcommand{\Homs}[2]{\mathrm{Hom}^s_{\Fq}{(#1,#2)}}
\def\res#1{\raise-.5ex\hbox{\ensuremath|}_{#1}}

\newcommand{\widesim}{
  \mathrel{{\scalebox{1.5}[1]{$\sim$}}}
}

\newcommand{\reducesize}[2]{%
  \mathbin{
    \ooalign{
      \raisebox
       {.4ex}
          {$#1\widesim$}
      \cr 
      \hidewidth
      \raisebox
        {-.2ex}
        {\scalebox
          {.75}
          {$#1#2$}
        }
      \hidewidth
    }
  }
}

\newcommand{\stb}[1]{\mathpalette\reducesize{#1}}

\newcommand{\Fq}{\mathbb{F}_q}

\newcommand{\C}{\mathcal{C}}

\newcommand{\SL}{\mathrm{SL}}

\newcommand{\Or}{\mathrm{O}}

\newcommand{\Sp}{\mathrm{Sp}}
\newcommand{\POmega}{\mathrm{P\Omega}}
\newcommand{\PSL}{\mathrm{PSL}}

\newcommand{\PGL}{\mathrm{PGL}}
\newcommand{\PGammaL}{\mathrm{P\Gamma L}}

\newcommand{\GL}{\mathrm{GL}}

\newcommand{\RC}{\mathrm{RC}}
\newcommand{\base}{\mathrm{b}}
\newcommand{\Base}{\mathrm{B}}
\newcommand{\Height}{\mathrm{H}}
\newcommand{\Irred}{\mathrm{I}}
\newtheorem{prop}{Proposition}[section]
\newtheorem{thm}[prop]{Theorem}
\newtheorem{conj}[prop]{Conjecture}

\newtheorem{cor}[prop]{Corollary}
\newtheorem{lem}[prop]{Lemma}
\newtheorem{defn}[prop]{Definition}
\newtheorem{remark}[prop]{Remark}
\numberwithin{equation}{section}

\begin{document}
\title[Height and relational complexity]{On the height and relational complexity of a finite permutation group}

\author{Nick Gill}
\address{ Department of Mathematics, University of South Wales, Treforest, CF37 1DL, U.K.}
\email{nick.gill@southwales.ac.uk}

\author{Bianca Lod\'a}
\address{ Department of Mathematics, University of South Wales, Treforest, CF37 1DL, U.K.}
\email{bianca.loda@southwales.ac.uk}

\author{Pablo Spiga}
\address{Dipartimento di Matematica e Applicazioni, University of Milano-Bicocca, Via Cozzi 55,  20125 Milano, Italy}
\email{pablo.spiga@unimib.it}


\begin{abstract}
Let $G$ be a permutation group on a set $\Omega$ of size $t$. We say that $\Lambda\subseteq\Omega$ is an \emph{independent set} if its pointwise stabilizer is not equal to the pointwise stabilizer of any proper subset of $\Lambda$. We define the \emph{height} of $G$ to be the maximum size of an independent set, and we denote this quantity $\Height(G)$.

In this paper we study $\Height(G)$ for the case when $G$ is primitive. Our main result asserts that either $\Height(G)< 9\log t$, or else $G$ is in a particular well-studied family (the ``primitive large--base groups''). An immediate corollary of this result is a characterization of primitive permutation groups with large ``relational complexity'', the latter quantity being a statistic introduced by Cherlin in his study of the model theory of permutation groups. 

We also study $\Irred(G)$, the maximum length of an irredundant base of $G$, in which case we prove that if $G$ is primitive, then either $\Irred(G)<7\log t$ or else, again, $G$ is in a particular family (which includes the  primitive large--base groups as well as some others).
\end{abstract}


\maketitle

\section{Introduction}\label{s: intro}

In this paper we study a number of different statistics pertaining to primitive permutation groups. We are interested in understanding which families of primitive permutation groups exhibit large values for these various statistics.
  
From here on, let $G$ be a finite primitive permutation group on a set $\Omega$ of size $t<\infty$. The statistic of most interest to us is the \emph{relational complexity} of a permutation group, $\RC(G)$, a statistic that was first introduced in \cite{cherlin_martin}.  We were motivated by a remark in the same paper in which the authors suggest that it should be possible to classify those primitive groups $G$ for which $\RC(G)>\sqrt{t}$. Our main result (more or less) yields this classification; indeed it applies with the $\sqrt{t}$ replaced by the asymptotically weaker $9\log t+1$ (here, and everywhere, logarithms are base $2$).
  
In the course of our investigations we were inspired by the following result of Liebeck \cite{liebeck} concerning the minimum base size, $\base(G),$ of the permutation group $G$. (The definition of this quantity is given below.) 

\begin{thm}\label{t: liebeck}
Let $G$ be a finite primitive group of degree $t$. Then one of the following holds:
\begin{enumerate}
\item $G$ is a subgroup of $\Sym(m) \wr \Sym(r)$ containing $(\Alt(m))^r$, where the action of $\Sym{(m)}$ is on $k$-subsets of $\{ 1, \dots, m\}$ and the wreath product has the product action of degree $t= \binom{m}{k} ^r$;
\item $\base(G) < 9 \log{t} $.
\end{enumerate}
\end{thm}

Theorem~\ref{t: liebeck} leads us to make the following definition.

\begin{defn}
The group $G$ is a \emph{primitive large--base group} if $G$ is a subgroup of $\Sym(m) \wr\Sym(r)$ containing $(\Alt(m))^r$, where the action of $\Sym{(m)}$ is on $k$-subsets of $\{ 1, \dots, m\}$ and the wreath product has the product action of degree $t= \binom{m}{k} ^r$.
\end{defn}

In this paper we prove two results which are variants of the main result of Theorem~\ref{t: liebeck}; in both, the family of primitive large--base groups appear in a similar way to Theorem~\ref{t: liebeck}, in that they exhibit exceptional behaviour with respect to certain statistics. In order to state these results we must first define the statistics of interest.\footnote{There has been recent improvement on Liebeck's result. We now know that if $G$ is not a primitive large-base group, then $ \base(G) \leq \max\{ \lceil \log t \rceil + 1, 7 \}.$ \cite{crd}.}
  
\subsection{Definition of statistics}  
  
For $\Lambda = \{\omega_1,\dots,\omega_k\} \subseteq \Omega$, we write
$G_{(\Lambda)}$ or $G_{\omega_1, \omega_2, \dots, \omega_k}$  for the
pointwise stabilizer. If $G_{(\Lambda)} = \{1\}$, then we say that $\Lambda$ is a \emph{base}. The size of a smallest possible base is known as the \emph{base size} of $G$ and is denoted $\base(G)$. 

We say that a base is a \emph{minimal base} if no proper subset of it is a base. We denote the maximum size of a minimal base by $\Base(G)$.

Given an ordered sequence of elements of $\Omega$, $[\omega_1,\omega_2,\dots, \omega_k]$, we can study the associated \emph{stabilizer chain}:
\[
  G \geq G_{\omega_1} \geq G_{\omega_1, \omega_2}\geq 
  G_{\omega_1, \omega_2, \omega_3} \geq \dots \geq
  G_{\omega_1, \omega_2, \dots, \omega_k}
\]
If all the inclusions given above are strict, then the stabilizer chain is called \emph{irredundant}. If, furthermore, the group $G_{\omega_1,\omega_2,\dots, \omega_k}$ is trivial, then the sequence $[\omega_1,\omega_2,\dots, \omega_k]$ is called an \emph{irredundant base}. The size of the longest possible irredundant base is denoted $\Irred(G)$. Note that an irredundant base is not a base (because it is an ordered sequence, not a set).

Finally, let $\Lambda$ be any subset of $\Omega$. We say that $\Lambda$ is an \emph{independent set} if its pointwise stabilizer is not equal to the pointwise stabilizer of any proper subset of $\Lambda$. We define the \emph{height} of $G$ to be the maximum size of an independent set, and we denote this quantity $\Height(G)$. 

There is a basic connection between the four statistics we have defined so far:
\begin{equation}\label{eq: inequality}
 \base(G) \leq \Base(G) \leq \Height(G) \leq \Irred(G)  \leq \base(G)\log t.
\end{equation}
The proof of \eqref{eq: inequality} goes as follows: The first inequality is obvious. For the second, suppose that $\Lambda=\{\omega_1,\dots,\omega_k\}$ is a minimal base; then observe that $\Lambda$ is also an independent set. For the third, suppose that $\Lambda=\{\omega_1,\dots,\omega_k\}$ is an independent set; then observe that the ordered list $[\omega_1,\dots, \omega_k]$ can be extended to form an irredundant base.

The fourth inequality has been attributed to Blaha \cite{blaha} who, in turn, describes it as an ``observation of Babai'' \cite{babai}. Suppose that $G$ has a base of size $b=\base(G)$. Then, in particular $|G|\leq t^b$. On the other hand, any irredundant base and any independent set have size at most $\log|G|$. We conclude that $\Height(G) \leq \log(t^b)$, and the result follows.

We have one more statistic to define. To start, suppose that $r,n\in \mathbb{N}$ with $r\leq n$. If $I,J\in \Omega^n$, then we write $I\stb{r} J$, and say that $I$ is \emph{$r$-equivalent} to $J$ with respect to the action of $G$, if $G$ contains elements that
map every subtuple of size $r$ in $I$ to the corresponding subtuple in
$J$ i.e. $$\textrm{for every } k_1, k_2, \dots, k_r\in\{ 1,
\ldots, n\}, \textrm{ there exists } h \in G \textrm{ with }I_{k_i}^h=J_{k_i}, \textrm{ for every }i \in\{
1, \ldots, r\}.$$ 
Here $I_k$ denotes the $k^{\text{th}}$ element of tuple $I$ and $I^g$ denotes the image of $I$ under the action of $g$. Note that $n$-equivalence simply requires the existence of an element of $G$ mapping $I$ to $J$.

The group $G$ is said to be of {\it relational complexity $r$} if $r$ is the smallest integer such that, for all $n\in\mathbb{N}$ with $n\geq r$ and for all $n$-tuples $I, J \in \Omega^n$, 
\[
 I\stb{r} J \, \, \Longrightarrow I\stb{n} J.
\]
In this case we write $\RC(G)=r$. The fact that relational complexity is well-defined takes a little proving. There is an equivalent definition of $\RC(G)$ in terms of relational structures: $\RC(G)$ is the least $k$ for which $G$ can be viewed as an automorphism group acting naturally on a homogeneous relational system whose relations are $k$-ary \cite{cherlin2}. 

The relationship between relational complexity and the statistics defined above is given by the following inequality (which is proved in \S\ref{s: lemmas}):
\begin{equation}\label{eq: inequality 2}
 \RC(G) \leq \Height(G) + 1.
\end{equation}

Note that in what follows, we may sometimes include the set on which we are acting in our statistical notation, if this set is in any doubt. So, for instance, $\Height(H)$ and $\Height(H,\Delta)$ both mean ``the height of the permutation group $H$ on a set $\Delta$''.

\subsection{Main results}  

Our first main result is a Liebeck-type result for height.

\begin{thm}\label{t: height} 
Let $G$ be a finite primitive group of degree $t$. If $G$ is not a primitive large--base group, then
\[
\Height(G) <9 \log{t}.
\]
\end{thm}
Equation~\eqref{eq: inequality 2} immediately yields the following corollary.

\begin{cor}\label{c: rc}
Let $G$ be a finite primitive group of degree $t$. If $G$ is not a primitive large--base group, then
\[
\RC(G) < 9 \log{t}+1. 
\]
\end{cor}
  
Corollary~\ref{c: rc} yields the classification proposed in \cite{cherlin_martin} in a very strong form (since our bound is logarithmic in $t$, whereas the original suggestion was for $\sqrt{t}$). Note, however, that we do not assert that all of the primitive large--base groups genuinely violate the bound in Corollary~\ref{c: rc}. A precise analysis of the relational complexity of the primitive large--base groups was started in \cite{cherlin_martin} but, although a great deal of progress was made, much remains to be done. Note, though, that \cite[Theorem~2]{cherlin2} implies that the family of primitive large--base groups contain an infinite number of groups violating the bound in Corollary~\ref{c: rc} including, in particular, the alternating groups in their natural action. It is clear that the same is true with respect to Theorem~\ref{t: height}: both $\Alt(t)$ and $\Sym(t)$ in their natural action violate the bound in Theorem~\ref{t: height} for $t$ large enough. In future work we will show that the same is true of $\Alt(t)$ and $\Sym(t)$ in their action on $k$-sets.

Equation~\eqref{eq: inequality} yields a second corollary.

\begin{cor}\label{c: B}
Let $G$ be a finite primitive group of degree $t$. If $G$ is not a primitive large--base group, then
\[
\Base(G) < 9 \log{t}. 
\]
 \end{cor}

Of course, we could get yet another corollary by replacing $\Base(G)$ by $\base(G)$ here, but this would only reprise Theorem~\ref{t: liebeck}. What, then, of $\Irred(G)$? Our main result concerning this statistic is the following.

\begin{thm}\label{t: irred}
 Let $G$ be a finite primitive group of degree $t$. Then one of the following holds:
 \begin{enumerate}
  \item There exists an almost simple group $A$, with socle $S$, such that $G$ is a subgroup of $A \wr \Sym(r)$ containing $S^r$, the action of $A$ is  one of the following:
  \begin{enumerate}
  \item the action of $\Sym{(m)}$ on $k$-subsets of $\{ 1, \dots, m\}$ (so degree $s=\binom{m}{k}$);
  \item the action of a classical group on a set of subspaces of the natural module, or on a set of pairs of subspaces;
  \end{enumerate}
   and the action of the wreath product has the product action of degree $t=s^r$, where $s$ is the degree of the action of $A$.
  \item $\Irred(G) <  7\log t$.
 \end{enumerate}
\end{thm}

The listed possible actions of $A$ are a subset of the so-called ``standard actions'' of almost simple groups. In the second case, let $V$ be a finite-dimensional vector space over $\Fq$; if $S=\PSL(V)$, then such an action is on all subspaces of some fixed dimension $k$ in $V$; if $S={\rm Cl}(V)$, one of the other classical groups defined on $V$, then such an action is either on all non-degenerate subspaces of some fixed dimension $k$ in $V$, or on  totally-isotropic subspaces of some fixed dimension $k$ and of a given type in $V$.

Theorem~\ref{t: irred} is not as strong as we would like -- we would like to eliminate the groups listed at item (1b), so that we end up with the same exceptional family as in the other main results. Indeed we conjecture that something like Theorem~\ref{t: height} should hold for $\Irred(G)$ also:

\begin{conj}\label{conj: I}
There exists a constant $C>0$ such that if $G$ is a finite primitive group of degree $t$ that is not a primitive large--base group, then
\[
\Irred(G) < C \log{t}. 
\]
 \end{conj}

 \subsection{Structure of the paper} 
 Our proofs of Theorems~\ref{t: height} and \ref{t: irred} make use of the O'Nan--Scott Theorem which divides the class of primitive permutation groups into various families. There are various ways of stating the O'Nan--Scott Theorem; in this paper we make use of the division into eight types which is described in \cite{praeger}. 
 
After proving some useful background lemmas in \S\ref{s: lemmas}, the work in \S\S\ref{s: regular normal}, \ref{s: diagonal} and \ref{s: product} when combined with the O'Nan--Scott Theorem reduces the problem of proving Theorem~\ref{t: irred} to a question about almost simple groups. This question is addressed in \S\ref{s: as I} where the proof of Theorem~\ref{t: irred} is completed.

The work in \S\ref{s: product}, when combined with Theorem~\ref{t: irred} reduces the problem of proving Theorem~\ref{t: height} to a question about almost simple groups. This question is addressed in \S\ref{s: as H} where the proof of Theorem~\ref{t: height} is completed.
 

\subsection{Acknowledgments}

The results in this paper are based on the PhD thesis of the second author. The authors would like to thank Colva Roney-Dougal and Dugald Macpherson who, in the course of examining the thesis, made many useful remarks which have helped to improve the current paper.

The first and third authors would like to acknowledge the support of EPSRC grant EP/R028702/1 over the course of this research.

\section{Some background lemmas}\label{s: lemmas}
 
In this section we collect together a number of lemmas, mostly about height, which will be useful later. We also introduce one final statistic: for a group $G$, we define $\ell(G)$ to be the length of the longest subgroup chain in $G$. It is clear that $\ell(G)$ is greater than or equal to $\Irred(G)$ (and, hence, all of the other statistics defined in \S\ref{s: intro}); it is equally clear that $\ell(G)\leq \log |G|$.
 
We start with a proof of \eqref{eq: inequality 2}

\begin{lem}\label{l: rc and h}
 $\RC(G) \leq \Height(G) + 1$.
\end{lem}

\begin{proof} 
Let $h = \Height(G)$ and consider a pair $(I,J)\in \Omega^n$ such that $I\stb{r} J$ with $r=h+1$. We must show that $I\stb{n} J$.

Observe that we can reorder the tuples without affecting their equivalence. Hence, without loss of generality, we can assume that
\[ G_{I_1} > G_{I_1,I_2}> \dots > G_{I_1, I_2, \dots, I_{\ell}}
\] for some $\ell \leq h$ and then this chain stabilizes, i.e. 
\[ G_{I_1, \dots, I_{\ell}} = G_{I_1, \dots, I_{\ell+j}}
\] for all $1 \leq j \leq n-\ell$.
From the assumption of ${(h+1)}$-equivalence it follows that there exists an element $g \in G$ such that $I_i ^g = J_i$ for all $1 \leq i \leq \ell$ and observe that the set of all such elements $g$ forms a coset of $G_{I_1, \dots, I_{\ell}}$. 

The assumption of ${(h+1)}$-equivalence implies, moreover, that for all $1 \leq j \leq n-\ell$ there exists $g_j \in G$ such that 
\[ \begin{cases}
I_i ^{g_j} = J_i, \quad \text{for} \quad 1\leq i \leq \ell; \\
I_{\ell+j}^{g_j} = J_{\ell+j}.
\end{cases}
\]
The set of all such elements $g_j$ forms a coset of $G_{I_1, \dots, I_\ell, I_{\ell+j}}$, which is, again, a coset of $G_{I_1, \dots, I_\ell}$. Indeed, since any coset of $G_{I_1,\dots, I_\ell}$ is defined by the image of the points $I_1,\dots, I_\ell$ under an element of the coset, we conclude that elements of the same coset of $G_{I_1, \dots, I_\ell}$ map $I_{\ell+j}$ to $J_{\ell+j}$ for all $1 \leq j \leq n-\ell$. In particular, $I\stb{n} J$, as required.
\end{proof}

The next two lemmas are little more than observations and require no proof.

\begin{lem}\label{l: subgroup}
Let $H \leq G$ and let $\Lambda \subseteq \Omega$ be an independent set with respect to $H$. Then $\Lambda$ is an independent set with respect to $G$. In particular, $\Height(H)\le \Height(G)$.
\end{lem}

\begin{lem}\label{l: stab}
 Let $\Lambda= \{\alpha_1, \dots, \alpha_n\}$. Then $\Lambda$ is independent if and only if $$G_{(\Lambda)} \lneqq G_{(\Lambda \setminus \{\alpha_i\})}$$ for all $i=1, \dots, n$.
\end{lem}

\begin{lem}\label{l: subset}
Let $\Lambda \subseteq \Omega$. Then there exists $\Gamma \subseteq \Lambda$ such that $\Gamma$ is independent and $G_{(\Lambda)}= G_{(\Gamma)}$.
\end{lem}
\begin{proof} If $\Lambda$ is an independent set then the claim is trivially true. So assume it is not. Then there exists a proper subset $\Lambda_0 \subset \Lambda$ such that $G_{(\Lambda_0)}= G_{(\Lambda)}$. If $\Lambda_0$ is independent then we are done, otherwise consider $\Lambda_0$ and repeat the same argument.
\end{proof}

The next three results give information about the behaviour of the statistics $\Height(G)$, $\Irred(G)$ and $\ell(G)$ when the permutation group $G$ is a direct product.

\begin{lem}\label{l: viva}
Let $A$ and $B$ permutation groups on $\Omega_1$ and $\Omega_2$, then $$\Height(A\times B,\Omega_1\times \Omega_2)\le \Height(A,\Omega_1)+\Height(B,\Omega_2).$$
\end{lem}
\begin{proof}
Let $\pi_A:A\times B\to A$ be the projection onto $A$ and observe that the kernel of $\pi_A$ is $\{1\}\times B$. Let $\Lambda=\{(\alpha_1,\beta_1),\ldots,(\alpha_k,\beta_k)\}$ be an independent set for $A\times B$. Set $G:=A\times B$ and, for every $j\in \{1,\ldots,k\}$, set $G_j:=G_{(\alpha_1,\beta_1)}\cap\cdots\cap G_{(\alpha_j,\beta_j)}$. 

Now consider the set $\Lambda_1 = \{\alpha_1, \dots, \alpha_k\} \subseteq \Omega_1$. Observe that, \textit{a priori}, this set may not be independent. However, by Lemma \ref{l: subset}, there exists a subset $\{j_1, \dots, j_s\} \subseteq \{1, \dots, k\}$ such that the set $\Delta_1 = \{\alpha_{j_1}, \dots, \alpha_{j_s} \}$ is independent and $A_{(\Delta_1)}= A_{(\Lambda_1)}$. After relabeling elements, we can assume $\Delta_1= \{\alpha_1, \dots, \alpha_s\}$ and observe that $A_{(\Delta_1)}= \pi_A(G_s)$. In particular, $s \leq \Height(A,\Omega_1)$. 

Now let $H := G_s$. It is clear that the set 
$\{(\alpha_{s+1},\beta_{s+1}),\ldots,(\alpha_k,\beta_k)\}$ is independent with respect to $H$. Moreover, for all $j \in \left\{s+1, \dots, k\right\}$, we have 

$$ \pi_A \left(\bigcap_{i=s+1}^{k}{H_{(\alpha_i,\beta_i)}}\right)=\pi_A\left(\bigcap_{\substack{i=s+1 \\ i \neq j}}^{k}{H_{(\alpha_i,\beta_i)}}\right) $$ 

and this implies that $$\bigcap_{i=s+1}^{k}{H_{(\alpha_i,\beta_i)}} \cap (\{1\}\times B) \lneqq \left(\bigcap_{\substack{i=s+1 \\ i \neq j}}^{k}{H_{(\alpha_i,\beta_i)}} \right)\cap (\{1\}\times B).$$ 

Therefore, by Lemma~\ref{l: stab}, the set $\Delta_2=\{\beta_{s+1}, \dots, \beta_k\}$ is an independent set for $B\cap H$ and hence $k-s \leq \Height(B \cap H,\Omega_2)$. As $\Height(B \cap H,\Omega_2) \leq \Height(B,\Omega_2)$, by Lemma~\ref{l: subgroup}, we have $$k \leq \Height(A, \Omega_1)+ \Height(B, \Omega_2)$$ and we conclude that $\Height(G, \Omega_1 \times \Omega_2) \leq \Height(A, \Omega_1)+ \Height(B, \Omega_2).$
\end{proof}

\begin{lem}\label{l: viva1}
Let $A$ and $B$ be non-identity permutation groups on $\Omega_1$ and $\Omega_2$, then $$\Irred(A\times B,\Omega_1\times \Omega_2)= \Irred(A,\Omega_1)+\Irred(B,\Omega_2)-1.$$
\end{lem}
\begin{proof}
Let $[\alpha_1,\ldots,\alpha_k]$ and $[\beta_1,\ldots,\beta_s]$ be irredundant bases for $A$ and $B$ acting on $\Omega_1$ and on $\Omega_2$ and with $k:=\Irred(A,\Omega_1)$ and $s:=\Irred(B,\Omega_2)$. As $A$ and $B$ are not the identity, $s,k\ge 1$. Now consider the ordered sequence
$$[(\alpha_1,\beta_1),(\alpha_1,\beta_2),\ldots,(\alpha_1,\beta_s),(\alpha_2,\beta_s),(\alpha_3,\beta_s),\ldots,(\alpha_k,\beta_s)].$$
Using the fact that $[\alpha_1,\ldots,\alpha_k]$ and $[\beta_1,\ldots,\beta_s]$ are irredundant bases, it follows that this sequence of elements of $\Omega_1\times\Omega_2$ is an irredundant base for $A\times B$ acting on $\Omega_1\times \Omega_2$. Thus
$$\Irred(A\times B,\Omega_1\times\Omega_2)\ge s+k-1=\Irred(A,\Omega_1)+\Irred(B,\Omega_2)-1.$$

Now, let $[(\alpha_1,\beta_1),\ldots,(\alpha_t,\beta_t)]$ be an irredundant base for $A\times B$ acting on $\Omega_1\times \Omega_2$ with $t=\Irred(A\times B,\Omega_1\times \Omega_2)$. Let $\mathcal{S}_A:=\{i\in \{2,\ldots,t\}\mid A_{\alpha_1,\ldots,\alpha_{i-1}}>A_{\alpha_1,\ldots,\alpha_{i-1},\alpha_{i}}\}$ and
$\mathcal{S}_B:=\{i\in \{2,\ldots,t\}\mid B_{\beta_1,\ldots,\beta_{i-1}}>B_{\beta_1,\ldots,\beta_{i-1},\beta_{i}}\}.$ 
As, for each $i\in \{2,\ldots,t\}$, we have 
$$
A_{\alpha_1,\ldots,\alpha_{i-1}}
\times 
B_{\beta_1,\ldots,\beta_{i-1}}=
\bigcap_{j=1}^{i-1}(A\times B)_{(\alpha_j,\beta_j)}>\bigcap_{j=1}^{i}(A\times B)_{(\alpha_j,\beta_j)}=A_{\alpha_1,\ldots,\alpha_{i}}
\times 
B_{\beta_1,\ldots,\beta_{i}},$$
we deduce $\mathcal{S}_A\cup\mathcal{S}_B=\{2,\ldots,t\}$.
Moreover, if $\mathcal{S}_A=\{i_1,i_2,\ldots,i_r\}$, then $[\alpha_1,\alpha_{i_1},\alpha_{i_2},\ldots,\alpha_{i_r}]$ is an irredundant base for $A$ acting on $\Omega_1$ and hence $\Irred(A,\Omega_1)\ge |\mathcal{S}_A|+1$. Similarly, $\Irred(B,\Omega_2)\ge |\mathcal{S}_B|+1$. Therefore,
\begin{align*}
\Irred(A,\Omega_1)+\Irred(B,\Omega_2)&\ge |\mathcal{S}_A|+|\mathcal{S}_B|+2\ge|\mathcal{S}_{A}\cup\mathcal{S}_B|+2=|\{2,\ldots,t\}|+2=t+1\\
&=\Irred(A\times B,\Omega_1\times\Omega_2)+1.
\end{align*}
\end{proof}

\begin{lem}\label{l: l dp}
 Let $A$ and $B$ be groups. Then $\ell(A\times B)\leq \ell(A)+\ell(B)$.
\end{lem}
\begin{proof}
Let $G=A\times B$ and let
$$G_\kappa<G_{\kappa-1}<\cdots<G_1<G_0=G$$
be a chain of subgroups of length $\kappa$. Let $\pi:G\to B$ be the projection onto $B$. Let $$\mathcal{I}:=\{i\in\{0,\ldots,\kappa-1\}\mid \pi(G_i)>\pi(G_{i+1})\}$$
and let $s:=|\mathcal{I}|$. We may write $\mathcal{I}=\{i_1,\ldots,i_s\}$ with $i_1<i_2<\cdots<i_s$. Now,
$$
\pi(G_{i_s+1})<\pi(G_{i_s})<\cdots<\pi(G_{i_2})<\pi(G_{i_1})\le B$$
is a strictly increasing chain of subgroups of $B$ and hence $s\leq \ell(B)$.

Let $j$ be an arbitrary element in $\{1,\ldots,s-1\}$. Define $i_{s+1}=\kappa$. By definition
\begin{equation}\label{eq:eqeqeq3a}\pi(G_{i_{j+1}})=\cdots=\pi(G_{i_j+3})=\pi(G_{i_j+2})=\pi(G_{i_j+1}).\end{equation} 
Since, by assumption, we have the strictly increasing sequence
$$G_{i_{j+1}}<\cdots<G_{i_j+3}<G_{i_j+2}<G_{i_j+1},$$
the first isomorphism theorem and~\eqref{eq:eqeqeq3a} give 
$$A\cap G_{i_{j+1}}<\cdots<A\cap G_{i_j+3}<A\cap G_{i_j+2}<A\cap G_{i_j+1}.$$
This increasing sequence consists of $i_{j+1}-i_j$ subgroups.

The argument in the previous paragraph can be applied for every $j\in \{1,\ldots,s-1\}$ and hence we obtain $s$ chains of strictly increasing sequences of subgroups of $A$ consisting of $i_1-i_2,i_2-i_3,\ldots,i_{s-1}-i_s$ terms. By sticking these strictly increasing sequences together we obtain a longer increasing sequence of subgroups of $A$ of length $\kappa$. This longer increasing sequence is not necessarily strictly increasing, however the only positions where an equality can occur are the positions where we attach two strictly increasing chains, that is, in the positions
$$A\cap G_{i_{j+1}}\le A\cap G_{i_{j+1}-1}.$$
Since the number of these positions is $s$ and since  $s\le \ell(B)$, we have a strictly increasing chain in $A$ of length $\kappa-\ell(B)$. We conclude that $\kappa-\ell(B)\leq \ell(A)$, and the result follows.
\end{proof}

\begin{lem}\label{l: chain}
Let $G \leq \Sym(\Omega)$ and let $N \unlhd G$. Then
\[ \Height(G) \leq \Height(N) + \ell(G / N)
\] 
and
\[\Irred(G)\le \Irred(N)+\ell(G/N).\]
\end{lem}


\begin{proof}
Let $\{\omega_1,\ldots,\omega_k\}$ be an independent set for $G$ in its action on $\Omega$. 

Let us consider the action of $N$ on $\Omega$. Observe that, \emph{a priori}, the set $\{\omega_1,\ldots,\omega_k\}$ is not independent for this action. However, by Lemma \ref{l: subset}, there exists a subset $J := \{j_1, \dots, j_s\} \subseteq \{1, \dots, k\}$, such that $\{\omega_{j_1}, \dots, \omega_{j_s}\}$ is independent for $N$ on $\Omega$ and  
\[ N \cap \left(\bigcap_{i=1}^{s}{G_{\omega_{j_i}}}\right) = N \cap \left(\bigcap_{i=1}^{k}{G_{\omega_{i}}}\right).
\]
After relabeling elements, we can assume $ J=\{1, \dots, s\}.$ In particular, this yields that
\[ s \leq \Height(N,\Omega).
\]
 Observe that, for all $j \in \{s, \dots, k\}$, we have
\begin{equation}\label{eq:1989}
N \cap \left(\bigcap_{i=1}^{s}{G_{\omega_{i}}}\right)= N \cap \left(\bigcap_{i=1}^{j}{G_{\omega_{i}}}\right)
\end{equation} 



Let $\pi \colon G \to G/N$ be the natural projection and 
let $I= \{s, \dots, k\}$. 
Suppose that there exist $i_1, i_2 \in I$, with $i_1 < i_2$, such that
\[ \pi \left( \bigcap_{j=1}^{i_1}{G_{\omega_j}} \right) = \pi \left( \bigcap_{j=1}^{i_2}{G_{\omega_j}} \right).
\] 
This equality, \eqref{eq:1989} and the first isomorphism theorem imply that
\[ \bigcap_{j=1}^{i_1}{G_{\omega_j}}=\bigcap_{j=1}^{i_2}{G_{\omega_j}}.
\] Therefore
\[\bigcap_{j=1}^{k}{G_{\omega_j}} = \bigcap_{j=1}^{i_2}{G_{\omega_j}} \cap \bigcap_{j=i_2+1}^{k}{G_{\omega_j}}= \bigcap_{j=1}^{i_1}{G_{\omega_j}} \cap \bigcap_{j=i_2+1}^{k}{G_{\omega_j}}.\]
Since $\{1, \dots, i_1\} \cup \{i_2+1, \dots, k\} \subsetneq \{1, \dots, k\}$, we have a contradiction with the independence of $\{\omega_1, \dots, \omega_k\}$ for the action of $G$ on $\Omega$. 
 We conclude that, for $i_1,i_2 \in I=\{s,\dots, k\}$, with $i_1 < i_2$, we have
\[\pi \left( \bigcap_{j=1}^{i_1}{G_{\omega_j}} \right) < \pi \left( \bigcap_{j=1}^{i_2}{G_{\omega_j}} \right).
\]
We conclude that
$\ell(G/N) \geq k-s = \Height(G) - \Height(N)$, as required.

\smallskip

We now prove that $\Irred(G)\le \Irred(N)+\ell(G/N)$. Let $[\omega_1,\ldots,\omega_s]$ be an irredundant base for $G$ with $s:=\Irred(G)$. Now, set $G_0:=G$ and, for $i\in \{1,\ldots,t\}$, $G_i:=\cap_{j=1}^iG_{\omega_j}$. In particular, we have a strictly decreasing chain 
$$G_0>G_1>\cdots>G_{s-1}>G_s=1.$$
Now, let $\mathcal{S}_1:=\{i\in \{0,\ldots,s-1\}\mid G_iN>G_{i+1}N\}$ and $\mathcal{S}_2:=\{i\in \{0,\ldots,s-1\}\mid G_i\cap N>G_{i+1}\cap N\}$. We claim that $\mathcal{S}_1\cup\mathcal{S}_2=\{0,\ldots,s-1\}$. We argue by contradiction and we let $i\in \{0,\ldots,s-1\}\setminus(\mathcal{S}_1\cup\mathcal{S}_2)$. This means that $G_iN=G_{i+1}N$ and $G_i\cap N=G_{i+1}\cap N$. Let $g\in G_i\subseteq G_{i+1}N$. Then, there exist $h\in G_{i+1}$ and $n\in N$ with $g=hn$. Thus 
$h^{-1}g=n\in G_i\cap N=G_{i+1}\cap N$. As $h\in G_{i+1}$, this yields $g\in G_{i+1}$. Therefore $G_{i}=G_{i+1}$, which is clearly a contradiction. 

We have $|\mathcal{S}_1|\le \ell(G/N)$. Now, let $\mathcal{S}_2=\{i_1,i_2,\ldots,i_r\}$ with $i_1<i_2<\ldots<i_r$ and consider the sequence $[\omega_{i_1+1},\omega_{i_2+1},\ldots,\omega_{i_r+1}]$. We claim that this is an irredundant sequence for $N$, from which it follows that $\Irred(N)\ge r=|\mathcal{S}_2|$. Suppose, by contradiction, that
$$N_{\omega_{i_1+1}}\cap \cdots \cap N_{\omega_{i_j+1}}=N_{\omega_{i_1+1}}\cap \cdots \cap N_{\omega_{i_j+1}}\cap N_{\omega_{i_{j+1}+1}}.$$
This gives
$$G_{\omega_{i_1+1}}\cap \cdots \cap G_{\omega_{i_j+1}}\cap N=G_{\omega_{i_1+1}}\cap \cdots \cap G_{\omega_{i_j+1}}\cap G_{\omega_{i_{j+1}+1}}\cap N.$$
We now intersect both sides of this equality with $G_{i_{j+1}}=\bigcap_{u=1}^{i_{j+1}}G_{\omega_u}$ and we obtain
$$G_{i_{j+1}}\cap N=(G_{i_{j+1}}\cap G_{\alpha_{i_{j+1}+1}})\cap N=G_{i_{j+1}+1}\cap N,$$
contradicting the fact that $i_{j+1}\in\mathcal{S}_2$.

Summing up, $$\Irred(N)+\ell(G/N)\ge |\mathcal{S}_1|+|\mathcal{S}_2|\ge |\mathcal{S}_1\cup\mathcal{S}_2|=|\{0,\ldots,s-1\}|=s=\Irred(G).$$
\end{proof}

\section{Groups with a regular normal subgroup}\label{s: regular normal}
In this section we prove Theorem~\ref{t: irred} under the supposition that $G$ contains a regular normal subgroup. In fact here we use a general argument, which also holds in the case of imprimitive groups containing a regular normal subgroup. The primitive groups containing a regular normal subgroup are the groups of type HA, TW, HS or HC in the language of \cite{praeger}.

The main result of this section is the following.

\begin{prop}\label{p: regular normal}
Let $G$ be a permutation group on a finite set $\Omega$ of size $t$. If $G$ contains a regular normal subgroup, then
\[ \Irred(G) \le \log t+1.
\]
\end{prop}
\begin{proof}
Suppose that $G$ has a regular normal subgroup $N$ and fix $\omega_0 \in \Omega$. The action of $G$ on $\Omega$ is permutation isomorphic to the ``affine'' action of $G$ on $N$, where the group $N$ acts on $N$ via translations and the group $G_{\omega_0}$ acts by group conjugation. In particular, if $n,v\in N$ and $x\in G_\omega$, then
$$v^{xn}=v^x \cdot n = x^{-1}vx\cdot n.$$
In what follows, we identify $\Omega$ with $N$. We let $[\omega_1,\ldots,\omega_k]$ be an irredundant base and we set $H:=G_{\omega_0}$. We may assume, without loss of generality, that $\omega_1=1\in N$.
Now, $H_{\omega_2}=\cent H {\omega_2}$ fixes $\omega_2$ and hence it fixes each element of $\langle \omega_2\rangle$. Similarly, $$H_{\omega_2,\omega_3}=\cent H{\omega_2}\cap \cent H{\omega_3}=\cent H{\langle \omega_2,\omega_3\rangle}.$$ Continuing in this way, we obtain a chain of subgroups of $H$,
\begin{align}\label{eq:eqeq1}
H&\ge H_{\omega_2}=\cent H{\omega_2}\ge H_{\omega_2,\omega_3}=\cent H{\langle \omega_2,\omega_3\rangle}\ge\cdots\\\nonumber
&\cdots\ge H_{\omega_2,\omega_3,\ldots,\omega_k}=\cent H{\langle \omega_2,\omega_3,\cdots,\omega_k\rangle},
\end{align}
and a chain of subgroups of $N$,
\begin{align}\label{eq:eqeq2}
\langle 1\rangle&\le \langle \omega_2\rangle\le\cdots\le\langle \omega_2,\ldots,\omega_k\rangle\le N.
\end{align}
Since $[\omega_1,\ldots,\omega_k]$ is an irredundant base, the inequalities in~\eqref{eq:eqeq1} are strict inequalities. This yields that the inequalities in~\eqref{eq:eqeq2} must also be strict. In particular,
\begin{equation*}k\le \log|N|+1.\qedhere
\end{equation*}
\end{proof}

\section{Diagonal action}\label{s: diagonal}

In this section we prove Theorem~\ref{t: irred} under the supposition that $G$ is a primitive group of type SD, see~\cite{praeger}. Our main result is the following.

\begin{prop}\label{p: diagonal}
Let $G$ be a primitive permutation group on a finite set $\Omega$ of size $t$. If $G$ is of type SD, then
\[ \Irred(G) \le\log t.
\]
\end{prop}

We start by reviewing the structure of primitive groups of diagonal type, this will also help us to set some notation.

Let $T$ be a non-abelian simple group, let $m$ be a positive integer with $m\ge 2$ and let $S:=T^m$ be the Cartesian product of $m$ copies of $T$. We denote by $D:=\{(t,\ldots,t)\mid t\in T\}$ the diagonal subgroup of $S$ and we set 
$$\Omega:= [ D :S ]$$
the set of right cosets of $D$ in $S$. Each element of $\Omega$ has a distinguished coset representative, that is, an element whose first coordinate begins with a $1$. In other words,
$$D(t_1,t_2,\ldots,t_m)=D(1,t_1^{-1}t_2,\cdots,t_1^{-1}t_m).$$
In particular, $\Omega$ is in one-to-one correspondence with the elements of $T^{m-1}$ and hence
\begin{equation}\label{eq: tt}
\log |\Omega|=(m-1)\log |T|.
\end{equation}
Observe, first, that the elements of $S$ act on $\Omega$ by coset multiplication, that is, for every $D(t_1,\ldots,t_m)\in\Omega$ and $(x_1,\ldots,x_m)\in S$ we have
$$D(t_1,\ldots,t_m)^{(x_1,\ldots,x_m)}=D(t_1x_1,\ldots,t_m x_m).$$
Observe, second, that the elements of $\mathrm{Aut}(T)$ act on $\Omega$ ``componentwise'', that is, for every $D(t_1,\ldots,t_m)\in\Omega$ and $\varphi\in \Aut(T)$ we have
$$D(t_1,\ldots,t_m)^\varphi=D(t_1^\varphi,\ldots,t_m^\varphi).$$
Two comments are in order. First, this does indeed define an action of $\Aut(T)$ on $\Omega$ because $D$ is $\Aut(T)$-invariant. Second, the inner automorphisms of $\Aut(T)$ induce on $\Omega$ permutations appearing in $S$. (Let us denote by $\iota_x$ the inner-automorphism of $T$ defined by $x\in T$.) Indeed, for every $D(t_1,\ldots,t_m)\in\Omega$ and $x\in T$, we have
\begin{align*}
D(t_1,\ldots,t_m)^{\iota_x}&=D(t_1^x,\ldots,t_m^x)=D(x^{-1}t_1x,\ldots,x^{-1}t_m x)\\
&=D(t_1x,\ldots,t_m x)=D(t_1,\ldots,t_m)^{(x,\ldots,x)}.
\end{align*}
Therefore, $\iota_x$ and $(x,\ldots,x)$ induce the same permutation on $\Omega$.

Observe, third, that $\Sym(m)$ acts on $\Omega$ by permuting the coordinates. Again, this action is well defined because $D$ is $\Sym(m)$-invariant. It is easy to see that $\Aut(T)$ and $\Sym(m)$ centralize each other and they normalize $S$. We define
$$W:=S(\Aut(T)\times \Sym(m))\cong T^{m}\cdot (\mathrm{Out}(T)\times \Sym(m)).$$

The group $W$ acts primitively on $\Omega$ and any subgroup $G$ of $W$ containing the socle $S$ and projecting primitively on $\Sym(m)$ is said to be a primitive group of diagonal type. With the notation just established we have the following.

\begin{lem}\label{l: diagonal}
Let $G$ be a primitive group of SD type with socle $T^m$, for some non-abelian simple group $T$ and for some integer $m\ge 2$. Then
 \[
  \Irred(G) \le
\begin{cases}
\frac{3m}{2} + \log |\Aut(T)|,&\textrm{when }m\ge 3,\\
\log|T|,&\textrm{when }m=2.
\end{cases}
 \]
\end{lem}
\begin{proof}
By Lemma~\ref{l: subgroup}, we can assume $G=W$. Set $\omega:=D(1,\ldots,1)$. A computation shows that  $$W_\omega=\Aut(T)\times \Sym(m).$$

Suppose first that $m\ge 3$. Lemma~\ref{l: l dp} implies that
\[
 \ell(W_\omega)\leq \ell(\Aut(T))+\ell(\Sym(m)).
\]
By \cite{cameron_solomon_turull}, we know that $\ell(\Sym(m))\le \frac{3m}{2}$. On the other hand $\ell(\Aut(T))\leq \log|\Aut(T)|-1$, where the $-1$ accounts for the fact that $|T|$ is divisible by at least two distinct odd primes. We conclude that
\begin{equation}\label{eq:eqeqeq4}
\Irred(G) \leq \ell(W_\omega)+1 \leq \frac{3m}{2}+\log |\Aut(T)|.
\end{equation}

Suppose next that $m=2$. We identify $\Omega$ with $T$. Set $H:=T^2\cdot \Out(T)$ and observe that $H$ has a regular normal subgroup and that $|W:H|=2$. Arguing as in the proof of Proposition~\ref{p: regular normal}, we deduce that
$$\Irred(H)\le \omega(|T|)+1,$$
where $\omega(|T|)$ is the number of prime divisors of $|T|$. Write $|T|=2^v o$, where $v,o\in \mathbb{N}$ and $o$ is odd. If $o\ge 43$, we have $\log_3(o)\le \log(o)-2$ and hence $\omega(|T|)\le \log(2^v)\log_3(o)\le \log|T|-2$ and hence $\Irred(H)\le \log|T|-1$. If $o<43$, then $T$ has a Sylow $2$-subgroup of index at most $41$ in $T$ and hence $T$ admits a faithful primitive permutation representation of degree at most $41$. Thus we have only a finite number of simple groups satisfying this property. A direct analysis yields that $T$ is either $\Alt(5)$ or $\PSL_3(2)$. When $T=\PSL_3(2)$, we have $\omega(|T|)=5<\log|T|-2$ and hence we obtain again $\Irred(H)\le \log|T|-1$. Therefore, except when $T=\Alt(5)$,  as $|G:H|=2$, we have $\Irred(G)\le \Irred(H)+1\le\log |T|$. When $T=\Alt(5)$, we have checked that $\Irred(G)=5\le \log|T|$.
 \end{proof}

\begin{proof}[Proof of Proposition~\ref{p: diagonal}]When $m=2$, the proof follows immediately from Lemma~\ref{l: diagonal}. Assume  that $m\ge 3$.
Lemma~\ref{l: diagonal} and \eqref{eq: tt} imply that it is sufficient to prove that
\[ \frac{3m}{2}+\log |\Aut(T)| \leq (m-1)\log |T|.
\]
We argue by contradiction and we suppose that this inequality does not hold.

 In \cite[Lemma 2.2]{quick} it is shown that, for a non-abelian simple group $T$
\[ \frac{|T|}{|\Out(T)|} \geq 30.
\] Now, for a centerless group, we have $|\Aut(T)| = |\Out(T)||T|$, hence $|\Aut(T)| \leq \frac{1}{30}|T|^2$. Thus
\[ \frac{3m}{2}+\log |\Aut(T)|\leq \frac{3m}{2} +2 \log |T| - \log 30.
\]
Therefore,
$$(m-1)\log|T|< \frac{3m}{2} +2 \log |T| - \log 30.$$
Rearranging the terms and dividing by $\log |T|$, we obtain
$$m-3<\frac{3m}{2\log |T|}-\frac{\log 30}{\log |T|}.$$
 An easy computation (using $|T|\ge 60$) shows that this is never satisfied.
\end{proof}

\section{Product actions}\label{s: product}

In the break-down described in \cite{praeger} there are three classes of groups left to deal with to prove Theorems~\ref{t: irred} and \ref{t: height}. In this section we deal with groups of type CD or PA. Note that our result for type CD wil be definitive, whereas our result for type PA will involve input from groups of type AS -- and these will be dealt with in the remainder of the paper. 

Note, furthermore, that we will need a result for groups of type PA that is specific to $\Irred(G)$, and another that is specific to $\Height(G)$. In all cases we let $H$ be a primitive non-regular group on $\Delta$ of type AS or SD and let $n$ be a positive integer with $n\ge 2$. We define $W:=H\mathrm{wr}\Sym(m)$ endowed with its primitive product action on the Cartesian product $\Omega:=\Delta^m$, that is, for every $(h_1 \dots, h_{m})\sigma \in W$ and $(\delta_1, \dots, \delta_{m}) \in \Omega$ we have
 
\[ 
\left(\delta_1, \dots, \delta_{m} \right)^{\left( h_1, \dots, h_{m}\right)\sigma}
= \left(\delta_1^{h_1}, \dots, \delta_{m}^{h_{m}}\right)^{\sigma}=
\left( \delta_{1^{\sigma^{-1}}}^{h_{1^{\sigma^{-1}}}}, \dots, \delta_{m^{\sigma^{-1}}}^{h_{m^{\sigma^{-1}}}} \right).
\]
Let $t := |\Omega|= |\Delta|^{m}$ and let $\pi:W\to\Sym(m)$ be the natural projection. Observe that the kernel of $\pi$ is the base group of $W$, that is, $H^m$. 

With the notation just established, the result we need is the following.

\begin{prop}\label{p: I CD and PA}
Let $G$ be a primitive permutation group on a finite set, $\Omega=\Delta^m$, of size $t$. 
\begin{enumerate}
 \item\label{eq:eqeqeq1} If $H$ is of type SD, then $\Irred(G) < 2 \log t$.
 \item\label{eq:eqeqeq2} If $H$ is of type AS and $\Irred(H,\Delta) < C\log |\Delta|$, then $\Irred(G) < (C+\frac{3}{2\log |\Delta|})\log t$.
 \item\label{eq:eqeqeq3} If $H$ is of type AS and $\Height(H,\Delta) < C\log |\Delta|$, then $\Height(G) < (C+\frac{3}{2\log |\Delta|})\log t$.
\end{enumerate}
 \end{prop}
\begin{proof}
We use, once again, the fact that a subgroup chain in $\Sym(m)$ has length at most $\frac32m$, i.e. $\ell(\Sym(m))\leq \frac32m$ \cite{cameron_solomon_turull}.

We start with~\eqref{eq:eqeqeq1}: From Lemma~\ref{l: chain}, we have $$\Irred(W,\Omega)\le \Irred(H^m,\Delta^m)+\ell(W/H^m)=
\Irred(H^m,\Delta^m)+\ell(\Sym(m))\le \Irred(H^m,\Delta^m)+\frac{3m}{2}.$$Now, Lemma~\ref{l: viva1} yields 
$\Irred(G,\Omega)\le m\Irred(H,\Delta)+3m/2$. Thus
\begin{equation}\label{e: chain}
 \Irred(W,\Omega) \leq m\Irred(H,\Delta) +\frac{3}{2}m.
\end{equation}

Let $G \leq W$ be primitive. Then we obtain that
\[ \Irred(G,\Omega)\le m \Irred(H,\Delta)+\frac{3m}{2}.
\]

If $H$ is of type SD, then Proposition~\ref{p: diagonal} implies that $\Irred(H, \Delta) \le\log |\Delta|$ and hence
\begin{align*} 
\Irred(G) &\le m \log|\Delta| + \frac{3}{2}m < 2m \log|\Delta|=2\log t,
\end{align*}
where the second inequality follows with a computation using $|\Delta|\ge 60$.

For (2) and (3), recall that if $H$ is of type AS, then $G$ is of type PA. Now, for item (2) we know, by supposition, that $\Irred(H) <C \log t$. This fact combined with \eqref{e: chain} yields the result. For item (3) the argument is the same, provided all occurrences of $\Irred(X,Y)$ are replaced with $\Height(X,Y)$ (for varying $X$ and $Y$) and replacing Lemma~\ref{l: viva1} with  Lemma~\ref{l: viva}.
\end{proof}

\section{Almost simple groups: Bounds for \texorpdfstring{$\Irred(G)$}{I(G)}}\label{s: as I}

Here we deal with groups ``of type AS'' in the language of \cite{praeger}; in other words, we study almost simple primitive permutation groups. The main result of this section is the following.

\begin{prop}\label{p: as}
 Let $G$ be a primitive almost simple permutation group on a set, $\Omega$, of size $t$. 
Either $\Irred(G)\le 6 \log t$ or else $G$ is one of the groups listed at (1) in Theorem~\ref{t: irred}.
\end{prop}

The proof of Proposition~\ref{p: as} splits into several parts. To start with we use well-known results bounding $\base(G)$ to deal with so-called ``non-standard actions''. The terminology below follows \cite{burness_classical}.

\begin{defn}\label{d: subspace action}
Let $G$ be a classical group with socle $G_0$, and associated natural module $V$. A subgroup $H$ of $G$ not containing $G_0$ is a \emph{subspace subgroup} if for each maximal subgroup $M$ of $G_0$ containing $H \cap G_0$ one of the following holds:
\begin{enumerate}
\item $M$ is the stabilizer in $G_0$ of a proper non-zero subspace $U$ of $V$, where $U$ is totally singular, non-degenerate or, if $G_0$ is orthogonal and $p=2$, a non-singular $1$-space ($U$ can be any subspace if $G_0=\PSL(V)$).
\item $G_0= \Sp_{2m}(q)$, $p=2$ and $M= \Or^{\pm}_{2m}(q)$.
\end{enumerate}
A \emph{subspace action} of the classical group $G$ is the action of $G$ on the coset space $\left[G:H \right]$, where $H$ is a subspace subgroup of $G$.
\end{defn}

Note that the definition above amounts precisely to this: a maximal subgroup of $G$ is a subspace subgroup if it lies in any $\mathcal{C}_1$ class, or is the even--characteristic symplectic case in the $\mathcal{C}_8$ class. This definition requires that we follow \cite{kl} in labeling the classes $\mathcal{C}_1-\mathcal{C}_8$. A small extra collection of maximal subgroups arises when $G_0=\Sp_4(2^a)$ and $G$ contains a graph automorphism, or if $G_0=\POmega_8^+(q)$ and $G$ contains a triality graph automorphism. We note that \cite{kl} explicitly exclude these cases.

\begin{defn}
A transitive action of $G$ on a set $\Omega$ is said to be \emph{standard} if, up to equivalence of actions, one of the following holds:
\begin{enumerate}
\item $G_0 = \Alt(m)$ and $\Omega$ is an orbit of subsets or uniform partitions of $\left\{1, \dots, m \right\}$.
\item G is a classical group in a subspace action.
\end{enumerate}
\end{defn}

For an almost simple primitive permutation group in a non-standard action, the base size is bounded by an absolute constant. This was conjectured by Cameron and Kantor (\cite{cameron_conj},\cite{cameron_kantor}) and then settled in the affirmative by Liebeck and Shalev in \cite[Theorem 1.3]{liebeck_shalev}. The constant was then made explicit in subsequent work  \cite{burness_guralnick_saxl, burness_classical, burness_sporadic, burness_liebeck_shalev}. The following theorem summarizes these results.

\begin{thm}\label{t: base}
Let $G$ be a finite almost simple group in a primitive faithful non-standard action. Then $\base(G) \leq 7$, with equality if and only if $G$ is the Mathieu group $\mathrm{M}_{24}$ in its natural action of degree 24. 
\end{thm}

Theorem~\ref{t: base} and \eqref{eq: inequality} immediately yield Proposition~\ref{p: as} for non-standard actions.

\begin{lem}\label{l: non-standard}
Let $G$ be a finite almost simple permutation group on a set $\Omega$ of size $t$, in a non-standard action. Then
\[ \Irred(G) \leq 6 \log t.
\]
\end{lem}

One can compute that $\Irred(M_{24})=8$ (where we consider the natural action of $M_{24}$ on $24$ points), hence we have a constant ``6'' in the statement of Lemma~\ref{l: non-standard} rather than ``7''.

\subsection{Standard actions of \texorpdfstring{$A_n$ and $S_n$}{An and Sn}}

Let $G$ be $\Sym(n)$ or $\Alt{(n)}$. We must prove Theorem~\ref{t: irred} for the action of $G$ on partitions of $n$. Let $n=ab$ for some positive integers $a,b$ with $a \geq 2$ and $b \geq 2$. We denote by $\Omega_a^b$ the set of all uniform partitions of $n$, with $a$ parts of size $b$. Let $t = | \Omega_a^b |$, then
\[ t = \frac{(ab)!}{a! (b!)^a}.
\]
We consider the action of $G$ on $\Omega_a^b$. We have the following result:

\begin{lem}\label{l: an partition}
Let $G$ be an almost simple group with socle $\Alt(n)$ acting on $\Omega=\Omega_a^b$. Then
\[ \Irred(G) < 2 \log t.
\] 
\end{lem}

\begin{proof}
By \cite[Lemma 5.6]{gmps} we know that if $a \geq 2$, $b \geq 2$ and $n \geq 17$ then
\[
\frac{(ab)!}{(b!)^a(a!)} \geq 3^{ab/2}.
\]
Hence
\begin{equation}\label{eq:1}
\log t \geq \frac{\log (3)}{2}ab.
\end{equation}

Once again, we use the fact that $\ell(\Sym(n))\leq \frac32n=\frac32ab$.
As $3/2<\log(3)$ for $n\geq 17$, we have
\[ \Irred(G)\leq \ell(G) \leq \frac{3}{2}ab < \log(3)ab\le 2\log t
\] and the result follows. If $n<17$, we check directly that $\frac{3}{2} ab < 2 \log t$, unless $$(a,b)\in \{( 2, 3 ),
( 3, 2 ),
( 4, 2 ),
( 5, 2 ),
( 6, 2 )\}.$$
For these remaining cases, we have computed explicitly the value of $\Sym(n)$ acting on partitions of $\{1,\ldots,n\}$ in  $b$ parts of cardinality $a$ and we have verified that in each case $\Irred(G)<2\log t$.
\end{proof}

\subsection{The symplectic/ orthogonal case}

In this section we will deal with the actions listed at item (2) of Definition~\ref{d: subspace action}. In particular $G$ is almost simple with socle $G_0=\Sp_{2m}(q)$, $q$ a power of $2$. Consulting \cite{bhr, kl}, it is clear that if $H$ is a maximal subgroup of $G$ not containing $G_0$, and $H\cap G_0$ is a subgroup of $M=\Or_{2m}^\pm(q)$, then $H\cap G_0=M$. In light of this the result that we need is the following.

\begin{lem}\label{l: symp orth}
 Let $G$ be almost simple with socle $G_0=\Sp_{2m}(q)$, $q$ a power of $2$. Let $H$ be a subgroup of $G$ such that $H\cap \Sp_{2m}(q)=\Or_{2m}^\pm(q)$, let $\Omega$ be the set of cosets of $H$ in $G$, and write $t=|\Omega|$. Then
 \[
  \Irred(G,\Omega)<\frac{11}{3}\log t.
 \]
\end{lem}

The treatment that follows is inspired by \cite[\S 7.7]{dixon_mortimer}, where the case of $\Sp_{2m}(2)$ is considered. Let $e$ and $m$ be positive integers, let $q:=2^e$, let $\mathbb{F}_q$ be the finite field with $q$ elements, and let $V:=\mathbb{F}_q^{2m}$ be the $2m$-dimensional vector space of row vectors over $\mathbb{F}_q$. To start with we adjust notation slightly, and assume that $G$ is simple: let $G:=\Sp_{2m}(q)$ be the symplectic group defined by the symmetric matrix
\[
f:=\begin{pmatrix}
0&I\\
I&0
\end{pmatrix},
\] 
where $0$ and $I$ are the zero and identity $m\times m$-matrices, respectively. In particular, $G$ is the group of invertible matrices preserving the bilinear form $\varphi:V\times V\to \mathbb{F}_q$ defined by $$\varphi(u,v):=ufv^{T},$$ for every $u,v\in V$, that is
\[ G = \left\{g \in \GL_{2m}(q) \mid gfg^{T}=f\right\}.
\]
Note that the bilinear form $\varphi$ is alternating, i.e. for all $u \in V$, we have
\begin{equation}\label{eq:blabla-1}
\varphi(u,u)=0.
\end{equation}
Moreover, since $\Fq$ is of characteristic 2, the form $\varphi$ is symmetric, i.e. for all $u,v \in V$ ,we have
\begin{equation}\label{eq:blabla-3}
\varphi(u,v)= \varphi(v,u).
\end{equation}
Now we let $\Omega$ be the set of quadratic forms $\theta:V\to\mathbb{F}_q$ polarising to $\varphi$. Recall that this means that $\theta:V\to\mathbb{F}_q$ is a function satisfying
\begin{itemize}
\item $\theta(u+v)-\theta(u)-\theta(v)=\varphi(u,v)$, for every $u,v\in V$, and
\item $\theta(cu)=c^2u$, for every $c\in\mathbb{F}_q$ and $u\in V$.
\end{itemize}

Next, consider the matrix 
\[
e:=\begin{pmatrix}
0&I\\
0&0
\end{pmatrix}
\]
and the quadratic form $\theta_0:V\to\mathbb{F}_q$ defined by $$\theta_0(u):=ueu^T,$$
for every $u\in V$. For every $u,v\in V$, we have
\begin{equation}\label{eq:blabla-2}
\begin{split}
\theta_0(u+v)-\theta_0(u)-\theta_0(v)&:=(u+v)e(u+v)^T-ueu^T-vev^T\\
 &=ueu^T+vev^T+uev^T+veu^T-ueu^T-vev^T\\
 &=uev^T+veu^T=uev^T+ue^Tv^T=u(e+e^T)v^T=ufv^T\\
&=\varphi(u,v). 
\end{split}
\end{equation}

In particular, \(\theta_0\) is a quadratic form whose polarisation is the symplectic form $\varphi$ and hence $\theta_0\in \Omega$.

Let $\theta\in \Omega$ and define $\lambda:=\theta-\theta_0$. We have
\begin{align*}
\lambda(u+v)=&\theta(u+v)-\theta_0(u+v)=\theta(u)+\theta(v)+\varphi(u,v)-\theta_0(u)-\theta_0(v)-\varphi(u,v)\\
=&\lambda(u)+\lambda(v),\\
\lambda(cu)=&\theta(cu)-\theta_0(cu)=c^2\theta(u)-c^2\theta_0(u)=c^2\lambda(u),
\end{align*}
for every $u,v\in V$ and for every $c\in\mathbb{F}_q$. Therefore, since $\Fq$ is of characteristic $2$, the function $\lambda:V\to\mathbb{F}_q$ is semilinear and hence there exists a unique $b\in V$ such that
$\lambda(u)=(u\cdot b^T)^2$, for every $u\in V$ (see Lemma \ref{lemma1} for a precise statement). Since $f$ is an invertible matrix, there exists a unique $a\in V$ with $b=af$ and hence
$$\lambda(u)=(ufa^T)^2=\varphi(u,a)^2,$$
for every $u\in V$. Summing up, we have shown that an arbitrary element of $\Omega$ is of the form
$$u\mapsto \theta_0(u)+\varphi(u,a)^2,$$
where $a\in V$. We denote this element of $\Omega$ simply by $\theta_a$. Thus
\begin{equation}\label{eq:blabla}\theta_a(u)=\theta_0(u)+\varphi(u,a)^2,\quad\textrm{for every }u\in V.\end{equation} 
In particular, the elements of $\Omega$ are parametrised by the vectors of $V$. Moreover, if $\theta_{a}=\theta_{a'}$ for some $a, a' \in V$, then $\theta_{a}(u)=\theta_{a'}(u)$ for every $u \in V$ and this implies $\varphi(u, a)= \varphi(u, a')$ for every $u \in V$. Since $\varphi$ is non-degenerate, we obtain $a=a'$. Hence, the set $\Omega$ is in one-to-one correspondence with $V$. This, in particular, yields that $|\Omega|= q^{2m}$.

\begin{lem}
The group $G$ acts on the set $\Omega$.
\end{lem}
\begin{proof}
First, we show that for every $x\in G$ and for every $\theta\in \Omega$, the mapping 
\begin{equation}
\begin{split}\label{eq:blabla1}
\theta^x &\colon V\to\mathbb{F}_q\\
& u \mapsto \theta(ux^{-1})
\end{split}
\end{equation}
 gives rise to an element of $\Omega$. For every $u,v\in V$, we have
\begin{align*}
\theta^x(u+v)&=\theta((u+v)x^{-1})=\theta(ux^{-1}+vx^{-1})=\theta(ux^{-1})+\theta(vx^{-1})+\varphi(ux^{-1},vx^{-1})\\
&=\theta(ux^{-1})+\theta(vx^{-1})+\varphi(u,v)=\theta^{x}(u)+\theta^x(v)+\varphi(u,v).
\end{align*}
(Observe that in the fourth equality we used the fact that $x\in G$ and hence $x$ preserves the bilinear form $\varphi$.) Moreover,
\begin{align*}
\theta^x(cu)&=\theta((cu)x^{-1})=\theta(c(ux^{-1}))=c^2\theta(ux^{-1})=c^2\theta^x(u).
\end{align*}
Therefore, $\theta^x\in \Omega$. Finally, for every $x,y\in G$, $\theta\in \Omega$ and $u\in V$, we have
$$(\theta^{x})^y(u)=\theta^x(uy^{-1})=\theta((uy^{-1})x^{-1})=\theta(u(xy)^{-1})=\theta^{xy}(u).$$
Therefore, $(\theta^{x})^y=\theta^{xy}$ and hence $G$ defines a genuine right action on $\Omega$.
\end{proof}
Before continuing our discussion, we gather some information on $G$. Let $a\in V$, we define the mapping 

\begin{equation}\label{eq:blabla2}
\begin{split} 
t_a \colon &V\to V\\
&u \mapsto u+\varphi(u,a)a.
\end{split}
\end{equation}
Such a function is called a {\em transvection}. For every $u,v \in V$ and $c \in \Fq$ we have
\begin{align*}
(u+v)t_a & = (u+v) + \varphi(u+v,a)a = (u + \varphi(u,a)a) + (v + \varphi(v,a)a) = (u)t_a +(v)t_a; \\
(cu)t_a & = cu + \varphi(cu,a)a = c(u+\varphi(u,a)a)= c(u) t_a.
\end{align*}
Hence $t_a$ is linear. Moreover, for every $u\in V$ we have
\begin{align*}
(u)t_a^2&=(u+\varphi(u,a)a)t_a=u+\varphi(u,a)a+\varphi(u+\varphi(u,a)a,a)a\\
&=u+\varphi(u,a)a+\varphi(u,a)a+\varphi(u,a)\varphi(a,a)a=u+2\varphi(u,a)a+\varphi(u,a)\varphi(a,a)a\\
&=u+\varphi(u,a)\varphi(a,a)a \underset{\eqref{eq:blabla-1}}{=} u,
\end{align*}
where in the second-last equality we use the fact that the characteristic of $\Fq$ is $2$. This shows that $t_a$ is an involution. Finally, for every $u,v\in V$, we have
\begin{align*}
\varphi((u)t_a,(v)t_a)&=\varphi(u+\varphi(u,a)a,v+\varphi(v,a)a)\\
&=\varphi(u,v)+\varphi(v,a)\varphi(u,a)+\varphi(u,a)\varphi(a,v)+\varphi(u,a)\varphi(v,a)\varphi(a,a)\\
&\underset{\eqref{eq:blabla-3}}{=}\varphi(u,v) + 2 \varphi(v,a)\varphi(u,a) + \varphi(u,a)\varphi(v,a)\varphi(a,a) \\
&= \varphi(u,v) + \varphi(u,a)\varphi(v,a)\varphi(a,a)\\
&\underset{\eqref{eq:blabla-1}}{=}\varphi(u,v).
\end{align*}
Therefore $t_a$ preserves $\varphi$ and hence lies in the symplectic group $G$.

We are now interested in computing the image of $\theta_a$ under the transvection $t_c$. First recall that, in a field of characteristic 2, since $x= -x$ for every $x \in \Fq$, the square root $\sqrt{\cdot} \colon \Fq \to \Fq$ is a well-defined map. Moreover, for every $a,b,x,y \in \Fq$ such that $x=a^2$ and $y=b^2$ we have $(\sqrt{x}+\sqrt{y})^2= (a+b)^2= a^2+b^2= x+y$, which implies $\sqrt{x}+\sqrt{y}=\sqrt{x+y}$. Moreover, recall that $\theta_a$ is a quadratic form polarising to $\varphi$ and that $t_a$ is an involution, in particular $t_c=t_c^{-1}$. By using these facts, given $v\in V$, we have
\begin{align*}
\theta_a^{t_c}(u)&\underset{\eqref{eq:blabla1}}{=}\theta_a(ut_c^{-1})=\theta_a(ut_c)\underset{\eqref{eq:blabla2}}{=} \theta_a(u+\varphi(u,c)c)\\
&=\theta_a(u)+\theta_a(\varphi(u,c)c)+\varphi(u,\varphi(u,c)c)\\
&=\theta_a(u)+\varphi(u,c)^2\theta_a(c)+\varphi(u,c)^2\\
&=\theta_a(u)+(\theta_a(c)+1)\varphi(u,c)^2=\theta_a(u)+(\sqrt{\theta_a(c)+1}\varphi(u,c))^2\\
&=\theta_a(u)+\varphi(u,(\sqrt{\theta_a(c)}+1)c)^2\underset{\eqref{eq:blabla}}{=}\theta_0(u)+\varphi(u,a)^2+\varphi(u,(\sqrt{\theta_a(c)}+1)c)^2\\
&=\theta_0(u)+\varphi(u,a+(\sqrt{\theta_a(c)}+1)c)^2\\
&\underset{\eqref{eq:blabla}}{=}\theta_{a+(\sqrt{\theta_a(c)}+1)c}(u).
\end{align*}
From this, we deduce 
\begin{equation}\label{eq:0}
\theta_a^{t_c}=
\theta_{a+(\sqrt{\theta_a(c)}+1)c}.
\end{equation}

We now recall some facts about Galois theory. For a reference see \cite[Chapter VI]{lang}. The Frobenius mapping $\phi \colon x \mapsto x^2$ from $\mathbb{F}_q$ to itself is a generator of the Galois group of $\mathbb{F}_q$ over $\mathbb{F}_2$. There exists a well-defined $\mathbb{F}_2$-linear {\em trace} mapping $\mathrm{Tr}:\mathbb{F}_q\to\mathbb{F}_2$. In what follows, we need only two basic facts about $\mathrm{Tr}$: first, $\mathrm{Tr}$ is surjective and second, from Hilbert's 90 Theorem,  the kernel of $\mathrm{Tr}$ consists of the set $\{x^2+x\mid x\in\mathbb{F}_q\}$ and has cardinality $q/2$.

Define
\begin{align*}
\Omega^+:=\{\theta_a\mid \mathrm{Tr}(\theta_0(a))=0\},\\
\Omega^-:=\{\theta_a\mid\mathrm{Tr}(\theta_0(a))=1\}.
\end{align*}

Observe that the above definition is a generalization of the definition of $\Omega^+$ and $\Omega^-$ in \cite[Corollary 7.7 A]{dixon_mortimer}. Indeed, if $q=2$, then the Galois group is the trivial group and hence the trace map is the identity.


Let $N:=\langle t_a\mid a\in V\rangle$ be the subgroup of $G$ generated by the transvections. Observe that, for all $a,u \in V$ and $x \in G$, we have
\[ (u)x^{-1}t_ax= (ux^{-1}+\varphi(ux^{-1},a)a)x= u+\varphi(u,ax)ax = (u)t_{ax}.
\]In particular, this shows that $N\unlhd G$.
\begin{lem}\label{report} The sets $\Omega^+$ and $\Omega^-$ are $N$-orbits on $\Omega$, with
$$|\Omega^+|=\frac{q^m(q^m+1)}{2},\quad |\Omega^-|=\frac{q^m(q^m-1)}{2}.$$
These are also orbits for $G$.
\end{lem}
\begin{proof} We first prove that $$|\Omega^+|=\frac{q^m(q^m+1)}{2},\quad |\Omega^-|=\frac{q^m(q^m-1)}{2}.$$
Clearly, it suffices to prove the first equality because the second follows from the equality $$|\Omega^-|=|\Omega|-|\Omega^+|=q^{2m}-|\Omega^+|.$$ By definition, $\theta_0(a)=aea^T$. Moreover, by using the canonical basis of $V$, we have 
$$\mathrm{Tr}\left(\theta_0\left(\sum_{i=1}^{2m}a_ie_i\right)\right)=
\mathrm{Tr}\left(\sum_{i=1}^{m}a_ia_{i+m}\right)=\sum_{i=1}^m\mathrm{Tr}(a_ia_{i+m}).$$
Using this equation we may compute the cardinality of $|\Omega^+|$ arguing inductively on $m$. When $m=1$, we have $\mathrm{Tr}(a_1a_2)=0$. When $a_1=0$, we have $q$ solutions for $a_2$; however, for every $a_1\in \mathbb{F}_q\setminus\{0\}$, we have $q/2$ solutions for $a_2$. Therefore when $m=1$, we have $$q+(q-1)\frac{q}{2}=\frac{q^2}{2}+\frac{q}{2}=\frac{q(q+1)}{2}$$ solutions to the equation $\mathrm{Tr}(a_1a_2)=0$ and we have $\frac{q(q-1)}{2}$ solutions to the equation $\mathrm{Tr}(a_{1}a_2)=1$. Arguing inductively, we may assume that 
$$\sum_{i=1}^{m-1}\mathrm{Tr}(a_ia_{i+m})$$
is equal to 0 for $q^{m-1}(q^{m-1}+1)/2$ choices of $(a_1, \dots, a_{m-1})$ and is equal to 1 for $q^{m-1}(q^{m-1}-1)/2$ choices of $a_1, \dots, a_{m-1}$. Now, $$\sum_{i=1}^{m}\mathrm{Tr}(a_ia_{i+m})=\sum_{i=1}^{m-1}\mathrm{Tr}(a_ia_{i+m})+\mathrm{Tr}(a_ma_{2m})$$ has value $0$ if and only if $\sum_{i=1}^{m-1}\mathrm{Tr}(a_ia_{i+m})$ and $\mathrm{Tr}(a_ma_{2m})$ have the same value. Therefore altogether the number of solutions of $\sum_{i=1}^{m}\mathrm{Tr}(a_ia_{i+m})=0$ is
$$\frac{q^{m-1}(q^{m-1}+1)}{2}\frac{q(q+1)}{2}+\frac{q^{m-1}(q^{m-1}-1)}{2}\frac{q(q-1)}{2}=\frac{q^{m}(q^{m}+1)}{2}.$$

Next, we prove that $\Omega^+$ and $\Omega^-$ are $N$-orbits. We start by considering $\Omega^+$. We first prove that $\Omega^+$ is $N$-invariant. To this end, it suffices to show that if  $\theta_a\in \Omega^+$ and $c\in V$, then $\theta_a^{t_c}\in \Omega^+$. In other words, using~\eqref{eq:0}, if $\mathrm{Tr}(\theta_0(a))=0$, then $\mathrm{Tr}(\theta_0(a+(\sqrt{\theta_a(c)}+1)c))=0$. So let $\mathrm{Tr}(\theta_0(a))=0$ and recall that $\mathrm{Tr}$ is a linear map. We have

\begin{equation}\label{blabla1}
\begin{split} 
\mathrm{Tr}(\theta_0(a+(\sqrt{\theta_a(c)}+1)c))
&\underset{\eqref{eq:blabla-2}}{=}\mathrm{Tr}\left(\theta_0(a)+\theta_0((\sqrt{\theta_a(c)}+1)c)+\varphi(a,(\sqrt{\theta_a(c)}+1)c)\right)\\
&=\mathrm{Tr}\left(\theta_0(a)\right)+\mathrm{Tr}\left(\theta_0((\sqrt{\theta_a(c)}+1)c)\right)+\mathrm{Tr}\left(\varphi(a,(\sqrt{\theta_a(c)}+1)c)\right)\\
&=\mathrm{Tr}\left(\theta_0((\sqrt{\theta_a(c)}+1)c)\right)+\mathrm{Tr}\left(\varphi(a,(\sqrt{\theta_a(c)}+1)c)\right)\\
&=\mathrm{Tr}\left((\sqrt{\theta_a(c)}+1)^2\theta_0(c)\right)+\mathrm{Tr}\left(\varphi(a,(\sqrt{\theta_a(c)}+1)c)\right)\\
&=\mathrm{Tr}\left((\theta_a(c)+1)\theta_0(c)\right)+\mathrm{Tr}\left((\sqrt{\theta_a(c)}+1)\varphi(a,c)\right)\\
&=\mathrm{Tr}\left(\theta_a(c)\theta_0(c)\right)+\mathrm{Tr}\left(\theta_0(c)\right)+\mathrm{Tr}(\sqrt{\theta_a(c)}\varphi(a,c))+\mathrm{Tr}(\varphi(a,c)).
\end{split}
\end{equation}

Using $\theta_a(c)=\theta_0(c)+\varphi(a,c)^2$ and $\mathrm{Tr}(x^2)=\mathrm{Tr}(x)$ for every $x\in \mathbb{F}_q$, we obtain
\begin{equation}\label{blabla2}
\begin{split}
\mathrm{Tr}\left(\theta_a(c)\theta_0(c)\right)&=\mathrm{Tr}(\theta_0(c)^2)+\mathrm{Tr}(\varphi(a,c)^2\theta_0(c))\\ 
&=\mathrm{Tr}(\theta_0(c))+\mathrm{Tr}(\varphi(a,c)^2\theta_0(c)),
\end{split}
\end{equation}
\begin{equation}\label{blabla3}
\begin{split}
\mathrm{Tr}(\sqrt{\theta_a(c)}\varphi(a,c))&=\mathrm{Tr}(\sqrt{\theta_0(c)}\varphi(a,c))+\mathrm{Tr}(\varphi(a,c)^2)\\ 
 &=\mathrm{Tr}(\sqrt{\theta_0(c)}\varphi(a,c))+\mathrm{Tr}(\varphi(a,c)).
\end{split}
\end{equation}
Putting \eqref{blabla2} and \eqref{blabla3} into \eqref{blabla1}, we obtain 
\begin{align*}
\mathrm{Tr}(\theta_0(a+(\sqrt{\theta_a(c)}+1)c))= \,&\mathrm{Tr}(\theta_0(c))+\mathrm{Tr}(\varphi(a,c)^2\theta_0(c))+ \mathrm{Tr}\left(\theta_0(c)\right)\\
&+ \mathrm{Tr}(\sqrt{\theta_0(c)}\varphi(a,c))+\mathrm{Tr}(\varphi(a,c)) + \mathrm{Tr}(\varphi(a,c))\\
=&\,0.
\end{align*}

As $\Omega^-=\Omega\setminus\Omega^+$, we obtain that $\Omega^-$ is also $N$-invariant.

Next, we show that $\Omega^+$ is an $N$-orbit. Actually, we prove something stronger, we show that
$$\Omega^+=\{\theta_0^{t_c}\mid c\in V\}.$$
For every $a\in V$ with $\mathrm{Tr}(\theta_0(a))=0$, we need to show that there exists $c\in V$ such that
\begin{equation}\label{eq:11}
\theta_a=\theta_0^{t_c}.
\end{equation}
If $\theta_0(a)=0$, then we may take $c:=a$ and~\eqref{eq:0} yields
$$\theta_0^{t_a}=\theta_{0+(\sqrt{\theta_0(a)}+1)a}=\theta_{a}$$
and we are finished. Suppose $\theta_0(a)\ne 0$.
Since $\mathrm{Tr}(\theta_0(a))=0$, from Hilbert's 90 theorem, there exists $x\in \mathbb{F}_q$ with $\theta_0(a)=x+x^2$. Since $\theta_0(a)\ne 0$, we have $x\ne 0$.
Let $y\in \mathbb{F}_q\setminus\{0\}$ with $y^2=x$ and set $c:=y^{-1}a$. Thus
\begin{align*}
\sqrt{\theta_0(c)}+1&=\sqrt{\theta_0(y^{-1}a)}+1=\sqrt{y^{-2}\theta_0(a)}+1=y^{-1}\sqrt{\theta_0(a)}+1\\
&=y^{-1}\sqrt{x+x^2}+1=y^{-1}(y+y^2)+1=y.
\end{align*}
From~\eqref{eq:0}, we have
\begin{align*}
\theta_0^{t_c}
&=\theta_{0+(\sqrt{\theta_0(c)}+1)c}=\theta_{yc}=\theta_{yy^{-1}a}=\theta_a.
\end{align*}

Next, we show that $\Omega^-$ is an $N$-orbit. Actually, we prove something stronger, we show that
$$\Omega^-=\{\theta_\varepsilon^{t_c}\mid c\in V\}.$$ First, we select a distinguished element of $\Omega^-$. Let $\epsilon\in\mathbb{F}_q$ with $\mathrm{Tr}(\epsilon)=1$ and set $\varepsilon:=\epsilon e_1+e_{m+1}$, where $(e_i)_{i\in \{1,\ldots,2m\}}$ is the standard basis of $V$. Since $\theta_0(\epsilon e_1)=0=\theta_0(e_{m+1})$, we have $$\theta_0(\varepsilon)=\theta_0(\epsilon e_1)+\theta_0(e_{m+1})+\varphi(\epsilon e_1,e_{m+1})=\epsilon\varphi(e_1,e_{m+1})=\epsilon$$ and hence $\mathrm{Tr}(\theta_0(\varepsilon))=\mathrm{Tr}(\epsilon)=1$. Therefore, $\theta_\varepsilon\in \Omega^-$.

 For every $a\in V$ with $\mathrm{Tr}(\theta_0(a))=1$, we need to show that there exists $c\in V$ such that
\begin{equation}\label{eq:22}
\theta_a=\theta_\varepsilon^{t_c}.
\end{equation}
If $\theta_\varepsilon(a+\varepsilon)=0$, then we may take $c:=a+\varepsilon$ and~\eqref{eq:0} yields
$$\theta_\varepsilon^{t_{c}}=\theta_{\varepsilon+(\sqrt{\theta_\varepsilon(c)}+1)c}=\theta_{\varepsilon+c}=\theta_a$$
and we are finished. Suppose $\theta_\varepsilon(a+\varepsilon)\ne 0$. We have
\begin{align*}
\theta_\varepsilon(a+\varepsilon)&\underset{\eqref{eq:blabla}}{=}\theta_0(a+\varepsilon)+\varphi(a+\varepsilon,\varepsilon)^2= \theta_0(a+\varepsilon)+\varphi(a,\varepsilon)^2 + \varphi(\varepsilon,\varepsilon)^2\\
&\underset{\eqref{eq:blabla-1}}{=}\theta_0(a+\varepsilon)+\varphi(a,\varepsilon)^2\\
&\underset{\eqref{eq:blabla-2}}{=}\theta_0(a)+\theta_0(\varepsilon)+\varphi(a,\varepsilon)+\varphi(a,\varepsilon)^2.
\end{align*}
Since $\mathrm{Tr}(\theta_0(a))=1=\mathrm{Tr}(\theta_0(\varepsilon))$, by using the previous equality, we deduce that $$\mathrm{Tr}(\theta_\varepsilon(a+\varepsilon))=0.$$
From Hilbert's 90 theorem, there exists $x\in \mathbb{F}_q$ with $\theta_\varepsilon(a+\varepsilon)=x+x^2$. Since $\theta_\varepsilon(a+\varepsilon)\ne 0$, we have $x\ne 0$.
Let $y\in \mathbb{F}_q\setminus\{0\}$ with $y^2=x$ and set $c:=y^{-1}(a+\varepsilon)$. Thus
\begin{align*}
\sqrt{\theta_\varepsilon(c)}+1&=\sqrt{\theta_\varepsilon(y^{-1}(a+\varepsilon))}+1=\sqrt{y^{-2}\theta_\varepsilon(a+\varepsilon)}+1=y^{-1}\sqrt{\theta_\varepsilon(a+\varepsilon)}+1\\
&=y^{-1}\sqrt{x+x^2}+1=y^{-1}(y+y^2)+1=y.
\end{align*}
From~\eqref{eq:0}, we have
\begin{align*}
\theta_\varepsilon^{t_c}
&=\theta_{\varepsilon+(\sqrt{\theta_\varepsilon(c)}+1)c}=\theta_{\varepsilon+yc}=\theta_{\varepsilon+yy^{-1}(a+\varepsilon)}=\theta_a.
\end{align*}

The fact that $\Omega^-$ and $\Omega^+$ are also $G$-orbits follows from the fact that $N \unlhd G$ and $\left|\Omega^+\right|\ne \left|\Omega^-\right|$.
\end{proof}

We now compute the maximum length of a stabilizer chain in these two actions. Let $\epsilon\in \{+,-\}$ (we deal simultaneously with both cases). 
Let $\{\theta_{a_1},\theta_{a_2},\ldots,\theta_{a_k}\}$ be a subset of $\Omega^\epsilon$ such that the corresponding chain of stabilizers is strictly decreasing.
Without loss of generality, we may suppose that $a_1=0$ when $\epsilon=+$ and $a_1=\varepsilon$ when $\epsilon=-$. (Recall that $\varepsilon$ is an element of $V$ for which $\theta_\varepsilon$ is a distinguished element of $\Omega^-$; the definition of $\varepsilon$ was given in the proof of the previous lemma.) 

Let us define
$$\cent {G_{\theta_{a_1}}}{a_1+a_i}=\{x\in G_{\theta_{a_1}}\mid (a_1+a_i)x=a_1+a_i\},$$
that is the set of matrices in $G_{\theta{a_1}}$ fixing the vector $a_1+a_i\in V$.
\begin{lem}
For every $i\in \{2,\ldots,k\}$,
\begin{equation}\label{eq:-1}G_{\theta_{a_1}}\cap G_{\theta_{a_i}}=\cent {G_{\theta{a_1}}}{a_1+a_i},
\end{equation}
\end{lem}
\begin{proof} First observe that, given $j\in \{1,\ldots,k\}$, $x \in G_{\theta_{a_j}}$ if and only if for every $u \in V $ we have
\begin{equation}\label{eq:blabla20} \theta_0(ux^{-1})+\varphi(ux^{-1},a_j)^2= \theta_0(u)+\varphi(u,a_j)^2. 
\end{equation}
Fix $i\in \{2,\ldots,k\}$ and let $x \in G_{\theta_{a_1}} \cap G_{\theta_{a_i}}$. Then, by applying \eqref{eq:blabla20} with $j:=1$ and with $j:=i$ and by using the fact that the characteristic of $\Fq$ is $2$, we obtain that, for every $u \in V$,
\[ \varphi(ux^{-1}, a_1+a_i)^2 = \varphi(u, a_1+a_i)^2.
\] Since $\varphi$ is $G$-invariant, then $\varphi(ux^{-1}, a_1+a_i)^2 = \varphi(u, (a_1+a_i)x)^2$. Therefore, as $\varphi$ is non-degenerate, $(a_1+a_i)x=a_1+a_i$ and hence $x \in \cent {G_{\theta_{a_1}}}{a_1+a_i}$. 

Conversely, let $x \in \cent {G_{\theta_{a_1}}}{a_1+a_i}$. Then $x \in G_{\theta_{a_1}}$ and $(a_1+a_i)x= a_1+a_i$. We need to show that $x \in G_{\theta_{a_i}}$. Note that \eqref{eq:blabla20} applied with $j=1$ and the fact that we are in characteristic 2 imply that
\begin{equation}\label{eq:blabla30}
\theta_0(ux^{-1})= \theta_0(u)+\varphi(u,a_1)^2+\varphi(ux^{-1},a_1)^2.
\end{equation}
 For every $u \in V$ we have:
\begin{align*}
\theta_{a_i}^x(u) & \underset{\eqref{eq:blabla1}}{=} \theta_{a_i}(ux^{-1}) \underset{\eqref{eq:blabla}}{=} \theta_0(ux^{-1})+\varphi(ux^{-1},a_i)^2\\
&\underset{\eqref{eq:blabla30}}{=} \theta_0(u) +\varphi(ux^{-1},a_1)^2+\varphi(u,a_1)^2+\varphi(ux^{-1},a_i)^2 \\
&= \theta_0(u)+\varphi(ux^{-1},a_1+a_i)^2+\varphi(u,a_1)^2 \\
& =\theta_0(u)+\varphi(u,(a_1+a_i)x)^2+\varphi(u,a_i)^2 \\
& = \theta_0(u)+\varphi(u,a_1+a_i)^2+\varphi(u,a_1)^2 \\
&= \theta_0(u)+\varphi(u,a_1)^2+\varphi(u,a_i)^2 + \varphi(u,a_1)^2\\
&= \theta_0(u)+2 \varphi(u,a_1)^2+\varphi(u,a_i)^2\\
&= \theta_0(u)+\varphi(u,a_i)^2 \underset{\eqref{eq:blabla}}{=} \theta_{a_i}(u).
\end{align*}
Observe that in the sixth and seventh equalities we use, respectively, the fact that $x$ preserves $\varphi$ and the fact that $x$ fixes the vector $(a_1+a_i)$.
Hence $x \in G_{\theta_{a_i}}$ and therefore $x \in G_{\theta_{a_1}} \cap G_{\theta_{a_i}}$.
\end{proof}
For simplicity, for each $i\in \{1,\ldots,k\}$, write $G_i:=G_{\theta_{a_1}}\cap\cdots \cap G_{\theta_{a_i}}$ and, for each $i\in \{2,\ldots,k\}$, write $b_i:=a_1+a_i$. From~\eqref{eq:-1}, the strictly decreasing sequence
$$G_1>G_2>\cdots>G_{k-1}>G_k.$$
equals
\begin{align*}
G_1>&
\cent {G_1}{b_2}
=
\cent{G_1}{\spanq{b_2}}
>
\cent{G_1}{b_2,b_3}=
\cent{G_1}{\spanq{b_2,b_3}}>\cdots\\
>&
\cent {G_1}{b_2,\ldots,b_{k-1}}=\cent{G_1}{\spanq{ b_2,\ldots,b_{k-1}}}>
\cent {G_1}{b_2,\ldots,b_k}=\cent{G_1}{\spanq{ b_2,\ldots,b_k}}.
\end{align*}
Here, if $v_1,\dots, v_j\in V$, then we write $\spanq{v_1,\dots, v_j}$ to denote the $\Fq$-vector space generated by the $v_1,\dots, v_j$. We obtain that
$$0<\spanq{ b_2}<\spanq{ b_2,b_3}<\cdots<\spanq{ b_2,\ldots,b_k}\le V.$$

Observe that a strict inclusion in the above chain implies that the dimension has to go down by at least one at each step. Therefore $k\le 1+2m$. Thus, we have proved the following lemma.

\begin{lem}\label{l: sp stab} Let $G := \Sp_{2m}(q)$ act on $\Omega^{\epsilon}$, where $\epsilon \in \{+,-\}$. Then $\Irred(G, \Omega^{\epsilon}) \leq 1+2m$. 
\end{lem} 

We are now able to prove Lemma~\ref{l: symp orth} for $m\geq 3$. The proof is virtually identical when $m=2$, however to avoid annoying details, we will use the fact that when $m=2$ the result follows from Lemma~\ref{l: sp4} which we prove in a moment.

\begin{proof}[Proof of Lemma~\ref{l: symp orth} for $m\geq 3$]
Let $q=2^e$. Since $m\geq 3$, 
\[
 \mathrm{Sp}_{2m}(q) \unlhd G \leq \mathrm{\Gamma Sp}_{2m}(q) =\Sp_{2m}(q) \rtimes \C_e.
\]
Since a strictly descending chain of subgroups of $\C_e$ has length at most $\log e$, we conclude that
\[ \Irred(\mathrm{\Gamma Sp}_{2m}(q), \Omega^{\epsilon}) \le 2m +1 + \log e.
\]
Suppose $2m+1+\log e\ge 2\log t^\epsilon$. (Recall the cardinality of $\Omega^\epsilon$ from Lemma~\ref{report}.) Then
$$2^{2m+1}\cdot e \ge (t^\epsilon)^2= 2^{2em-2}\left(2^{em} +\epsilon 1\right)^2 \geq 2^{2em-2}\left(2^{em} - 1\right)^2.$$
When $e\ge 2$, with an easy computation we obtain a contradiction. When $e=1$, we have $2^{2m+1}\ge 2^{2m-2}(2^{m}-1)^2$, which implies $2^3\ge (2^m-1)^2$. Again this is a contradiction, because $m\ge 3$.

Putting these things together, we obtain that
$\Irred(G, \Omega^{\epsilon}) < 2\log t^{\epsilon}.$
\end{proof}

\subsection{Cases involving graph automorphisms}\label{s: graph}

We remarked after Definition~\ref{d: subspace action} that the almost simple groups with socle $\Sp_4(2^a)$ or $\POmega_8^{+}(q)$ containing a graph automorphism or a triality graph automorphism were not covered by the definition. We briefly deal with these groups here.

\begin{lem}\label{l: sp4}
 Let $G$ be an almost simple group with socle $G_0= \Sp_4(2^a)'$, and let $G$ act faithfully and primitively on $\Omega$, a set of size $t$. Then
 \[
  \Irred(G,\Omega)< \frac{11}{3}\log t.
 \]

\end{lem}
\begin{proof}
If $a=1$, then $G\leq \PGammaL_2(9)$. It is not hard to verify that the longest strictly increasing chain of subgroups of $\PGammaL_2(9)$ has length $6$. Thus $\Irred(G)\leq 6$. On the other hand $t\geq 6$ and so the result follows. Assume, then, that $a>2$. Note that 
\[ G_0 \leq G \leq \Aut(\Sp_4(2^a))= \mathrm{\Gamma Sp}_4(2^a) . \langle \gamma \rangle,
\] where $\gamma$ is a graph automorphism of order $2a$. 
We obtain that
\begin{equation}\label{eq:happy1}
\begin{split}
|G| &\leq  |\Aut(\Sp_4(2^a))| = |\Sp_4(2^a)|\cdot 2 \cdot a  \\
 &= 2^{4a}(2^{2a}-1)(2^{4a}-1) \cdot 2 \cdot a \leq 2^{11a}.
\end{split}
\end{equation}
On the other hand \cite[Theorem~5.2.2]{kl} implies that  
\begin{equation}\label{eq:happy2}
t \geq \frac{2^{4a}-1}{2^a-1}= 2^{3a}+2^{2a}+2^a+1 >2^{3a}. 
\end{equation}
Then \eqref{eq:happy1} and \eqref{eq:happy2} yield
\[
\Irred(G) \leq \log 2^{11a} = \log (2^{3a})^{\frac{11}{3}} = \frac{11}{3} \log 2^{3a} < \frac{11}{3} \log t.\qedhere
\]
\end{proof}

\begin{lem}\label{l: omega 8}
 Let $G$ be an almost simple group with socle $G_0=\POmega_8^+(q)$, where $q=p^f$ for some prime $p$ and positive integer $f$, and let $G$ act primitively on $\Omega$, a set of size $t$. Then
 \[
  \Irred(G,\Omega)< \frac{16}{3}\log t.
 \]
\end{lem}
\begin{proof}
 Note that
\[ G_0 \leq G \leq \Aut(\POmega_8^+(q))= \mathrm{P \Gamma O}_8^+(q). \langle \tau \rangle,
\] with $\tau$ a triality graph automorphism of order 3.
Now, using the fact that $6f<q^4$, and recalling that
\begin{align*}
& |G_0|= \frac{1}{d} q^{12}(q^4-1) \prod_{i=1}^{3}{(q^{2i}-1)},\\
& |\Out(G_0)|=6df,
\end{align*}
where $d=(4,q^4-1)$, we obtain that
\begin{equation}\label{eq:sad1}
\begin{split}
|G| &\leq |\Aut(G_0)|=|G_0|\cdot|\Out(G_0)|\\
&= 6df \cdot \frac{1}{d} q^{12}(q^4-1)(q^2-1)(q^4-1)(q^6-1) \\
&< 6fq^{28} < q^{32}.
\end{split}
\end{equation}
On the other hand \cite[\S 5.2]{kl} implies that
\begin{equation}\label{eq:sad2}
\begin{cases}
t \geq \frac{(q^4-1)(q^3+1)}{(q-1)} = (q^3+q^2+q+1)(q^3+1) > q^6, & \text{ for $q>2$ }; \\
t \geq 2^3(2^4-1) > 2^6, & \text{ for $q=2$}.
\end{cases}
\end{equation}
Then, \eqref{eq:sad1} and \eqref{eq:sad2} yield
\[ \Irred(G) \leq \log q^{32} = \log (q^6)^{\frac{32}{6}} < \frac{16}{3} \log t.\qedhere
\]
\end{proof}

Now Proposition~\ref{p: as} is a consequence of Lemmas~\ref{l: non-standard}, \ref{l: an partition}, \ref{l: symp orth}, \ref{l: sp4} and \ref{l: omega 8}. Similarly, Theorem~\ref{t: irred} is a consequence of Propositions~\ref{p: regular normal}, \ref{p: diagonal}, \ref{p: I CD and PA} and \ref{p: as} (observe that in Proposition~\ref{p: I CD and PA} we are using the fact that $|\Delta|\ge 5$ and hence $3/(2\log|\Delta|)<1$).

\section{Almost simple groups: Bounds on \texorpdfstring{$\Height(G)$}{H(G)}}\label{s: as H}

Now that Theorem~\ref{t: irred} is proved, we turn our attention to the proof of Theorem~\ref{t: height}. In light of Theorem~\ref{t: irred} and Proposition~\ref{p: I CD and PA}, all that is required is that we deal with the almost simple groups listed at item (1)(b) of Theorem~\ref{t: irred} -- we must show that these conform to the bound given in Theorem~\ref{t: height}. Thus the result that we need is the following.

\begin{prop}\label{p: as height}
Let $G$ be an almost simple group with socle, $G_0$, a simple classical group. Let $V$ be the associated natural module, of dimension $n$ over a field $\Fq$, let $m$ be an integer with $0<m<n$, let $\Omega$ be a set of $m$-dimensional subspaces of $V$ on which $G$ acts primitively and let $t=|\Omega|$. Then
\[
 \Height(G)<\frac{17}{2}\log t.
\]
\end{prop}

\subsection{\texorpdfstring{$G_0=\PSL_n(q)$}{G0=PSL(n,q)} and elements of \texorpdfstring{$\Omega$}{Omega} are \texorpdfstring{$m$-spaces}{m-spaces}}

In this case we suppose that $G_0=\PSL_n(q)$, and that $\Omega$ is the set of all $m$-dimensional subspaces of $V$, where $V$ is the natural $n$-dimensional module for $G_0$ over $\Fq$. Note that $G$ is transitive on $\Omega$. We prove the following result.

\begin{lem}\label{l: psl} 
Let $G$ be an almost simple group with socle $\PSL_n(q)$ acting on $\Omega$, the set of all $m$-spaces in $V$. Then 
\[ \Height(G) < \frac{11}{2}\log t.
\]
Moreover, if $q=p^f$ with $p$ a prime number and $G\le \mathrm{P}\Gamma\mathrm{L}_n(q)$, then
\begin{equation}\label{eq:197} 
\Height(G) <  2\min(m,n-m)n + \log \log_p{q}.
\end{equation}
\end{lem}

Using the fact that duality conjugates $m$-space stabilizers to $(n-m)$-space stabilizers, in what follows we may suppose that $m\le n/2$. Two notes concerning the actions considered in Lemma~\ref{l: psl}. First, it is easy to check that
\begin{equation}\label{eq: t bound}
 t> q^{m(n-m)}.
\end{equation}
Second, we note that if $m<\frac{n}{2}$, then  $G\leq \PGammaL_n(q)$. If $m=\frac{n}{2}$, then we must allow for the possibility that $G$ contains a graph automorphism of $\PSL_n(q)$.

\subsection{Some preliminaries}\label{preliminary}

Let $V$ and $W$ be finite-dimensional vector spaces over $\Fq$. We denote by $\Ends{V}$ the set of all semilinear transformations $V \to V$. Moreover, we write $\Hom{V}{W}$ for the set of all linear maps $V \to W$ and $\Homs{V}{W}$ for the set of all semilinear maps $V \to W$.

If $n=\dim(V)$, then we write $M_n(q)$ to denote the set of linear transformations $V\to V$; where a basis has been chosen, we also allow $M=M_n(q)$ to denote the set of all $n$-by-$n$ matrices over $\Fq$. If $W_1,\dots, W_k\leq V$, then we define
\[
 M_{W_1,\dots, W_k} = \{ g\in M \mid W_ig\leq W_i \textrm{ for all }i=1,\dots, k\}.
\]
Using this definition, it is natural to extend the concept of base, irredundant base and independent set in the context of the algebra $M=M_n(q)$. For instance, given a set $\Lambda:=\{W_1,\ldots,W_k\}$ of subspaces of $V$, we say that $\Lambda$ is an independent set for $M=M_n(q)$ if, for each $i\in \{1,\ldots,k\}$, we have
$$M_{W_1,\ldots,W_k}<M_{W_1,\ldots,W_{i-1},W_{i+1},\ldots,W_{k}}.$$

Recall that any element of $\aut(\Fq)$ induces an automorphism of $V$: we first fix a basis for $V$ and then act coordinate-wise. The following lemma is standard, but we include it for completeness.

\begin{lem}\label{lemma1}
Let $g \in \Homs{V}{W}$, $g \neq 0$. Then:
\begin{itemize}
\item[(a)] there exists a unique $\sigma \in \Aut(\Fq)$ such that $g$ is $\sigma$-semilinear. We say that $\sigma$ is the \emph{associated automorphism} of $g$;
\item[(b)] if $\sigma$ is the associated automorphism of $g$, then $g\sigma^{-1} \in \Hom{V}{W}$;
\item[(c)] there exists a unique $h \in \Hom{V}{W}$ such that $g= h\sigma$.
\end{itemize}
\end{lem}
\begin{proof}
\begin{itemize}
\item[(a)] Assume that there exist $\sigma_1,\sigma_2 \in \Aut(\Fq)$ such that $g$ is $\sigma_1$-semilinear and $\sigma_2$-semilinear. Hence for any $v \in V$ and $k \in \Fq$ we have $(kv)g=k^{\sigma_1}(vg)=k^{\sigma_2}(vg)$. In particular, as $g \neq 0$, there exists $v_0 \in V$ such that $(v_0)g \neq 0$. Therefore $k^{\sigma_1}(v_0 g)=k^{\sigma_2}(v_0 g)$, which implies $k^{\sigma_1}=k^{\sigma_2}$, for any $k \in \Aut{(\Fq)}$, and so $\sigma_1=\sigma_2$.
\item[(b)] Let $v_1, v_2 \in V$ and $k \in \Fq$, then
\begin{itemize}
\item $(v_1 + v_2)g \sigma^{-1}= (v_1)g \sigma^{-1} + (v_2)g \sigma^{-1}$;
\item $(kv_1)g \sigma^{-1}= (k^{\sigma}(v_1g))\sigma^{-1}= k((v_1)g \sigma^{-1})$.
\end{itemize}
Hence $g\sigma^{-1}$ is linear.
\item[(c)] From (b) it follows that there exists $h \in \Hom{V}{W}$ such that $g\sigma^{-1}=h$. Hence $g=h\sigma$. Assume that there exist $h_1,h_2 \in \Hom{V}{W}$ such that $g=h_1\sigma=h_2\sigma$. Then for any $v \in V$ we have $(v)h_1 \sigma=(v)h_2 \sigma$, which implies that $(v)h_1=(v)h_2$. So, we conclude that $h_1=h_2$.\qedhere
\end{itemize}
\end{proof}

Now let $V= W \oplus X$ for some $W,X \leq V$. Then every $v \in V$ can be written uniquely as $v=w+x$, for some $w \in W$ and $x \in X$. Let $f_W \in \Ends{W}$ and $f_X \in \Ends{X}$. We define $f_W \oplus f_X \in \Ends{V}$ as $v(f_W \oplus f_X) = wf_W + xf_X  $.

\begin{remark}\label{rmk7}
Let $V = W \oplus X$ for some $W,X \leq V$. Let $g \in \Ends{V}$. Then $g$ can be written as $g = g\res{W} \oplus g \res{X} $, provided $g(W) \leq W$ and $g(X) \leq X$.
\end{remark}

Finally, if $v_1,\dots, v_k \in V$, then we write $\spanq{v_1,\dots, v_k}$ to denote the $\Fq$-span of $v_1,\dots, v_k$.

\subsection{\texorpdfstring{$G\leq \PGL_n(q)$}{G<=PGL(n,q)}}\label{sec:G<=PGL(n,q)}

Let $V$ be $n$-dimensional over $\Fq$. In what follows we write $\M{n}$ for the set of all linear transformations $V\to V$. Note that it will be convenient to swap between thinking of $\GL_n(q)$ or $\PGL_n(q)$, depending on context -- this makes no difference to the calculations in question, as the center of $\SL_n(q)$ is the kernel of the action on $\Omega$.

The following lemma is crucial.

\begin{lem}\label{lemma6} Let $\SL_{n}{(q)} \leq G \subset \GL_n(q)$ act on $\Omega $. Let $\Lambda = \left\{W_1, \dots, W_\ell\right\} \subseteq \Omega$ be an independent set for the action of $G$ on $\Omega$, of maximal size. Then $$\dim \spanq{W_1, \dots, W_\ell} > n-m.$$ 
	\end{lem}
\begin{proof}
Let $W = \spanq{W_1, \dots, W_\ell} $ and suppose, first, that $k = \dim{W} = n-m$. Then we can assume, by transitivity, that $W= \spanq{e_1, \dots,  e_{n-m}}$.

Hence 
	\[ G_{(\Lambda)} \leq \left\{
	\left(\begin{array}{@{}c|c@{}}
	\begin{matrix}
	A
	\end{matrix}
	& \begin{matrix}
	0 
	\end{matrix}\\
	\hline
\begin{matrix} 
B
\end{matrix} &
\begin{matrix}
	C
	\end{matrix}
	\end{array} \right) \mid A \in \M{n-m}, B \in \M{m,n-m}, C \in \M{m} \right\} . \]
	
	Now we add to $\Lambda$ the following $m$-subspace
	\[ X= \spanq{e_{n-m+1}, \dots, e_n},
	\] and we denote $\bar{\Lambda}= \Lambda \cup \lbrace X \rbrace$. 
	Hence, by the maximality of $\Lambda$, we get
	\[ G_{(\bar{\Lambda})}= G_{(\bar{\Delta})},  \]
	for some $\bar{\Delta} \subseteq \bar{\Lambda}$. 
	
	We distinguish three cases:
	\begin{enumerate}
	\item[\textbf{Case I}]: if $\bar{\Delta} = \Lambda$, then $G_{(\bar{\Lambda})}=G_{(\Lambda)}$.  Now
	
	\[ G_{(\bar{\Lambda})} \leq  \left\{
	\left(\begin{array}{@{}c|c@{}}
	\begin{matrix}
	A
	\end{matrix}
	& \begin{matrix}
	0 
	\end{matrix}\\
	\hline
\begin{matrix} 
0
\end{matrix} &
\begin{matrix}
	C
	\end{matrix}
	\end{array} \right) \mid A \in \M{n-m}, C \in \M{m} \right\},
	\] but
	
	\[  G_{(\Lambda)} \geq \left\{
	\left(\begin{array}{@{}c|c@{}}
	\begin{matrix}
	I
	\end{matrix}
	& \begin{matrix}
	0 
	\end{matrix}\\
	\hline
\begin{matrix} 
B
\end{matrix} &
\begin{matrix}
	I
	\end{matrix}
	\end{array} \right) \mid  B \in \M{m,n-m}, B \neq 0 \right\},
	\] and this is a contradiction.
	
		\item[\textbf{Case II}]: if $\bar{\Delta} \subset \Lambda$, then $G_{(\Lambda)} \leq G_{(\bar{\Delta})}$. On the other hand, $ G_{(\bar{\Delta})} = G_{(\bar{\Lambda})} \leq G_{(\Lambda)} $, which implies $G_{(\Lambda)} = G_{(\bar{\Delta})}$, and hence a contradiction of the independence of $\Lambda$.
		\item[\textbf{Case III}]: if $\bar{\Delta} \subsetneq \Lambda$, then $X \in \bar{\Delta}$ and we denote $\Delta = \bar{\Delta} \setminus \lbrace X \rbrace$. It is sufficient to prove that $G_{(\Lambda)} = G_{(\Delta)}$, as this leads to a contradiction. Observe that $V = W \oplus X$ and so for any $g \in \End{V}$ we can define 
		\begin{align*} \pi_W \colon  \End{V} &\to \Hom{W}{V} \\
		g &\mapsto g\res{W} 
		\end{align*}
		and
		\begin{align*} \pi_X \colon \End{V} &\to \Hom{X}{V} \\
		g &\mapsto g\res{X} .
		\end{align*}
		
Now, if $g \in G_{(\bar{\Lambda})}$ or if $g \in G_{(\bar{\Delta})}$, then $g\res{W} \in \End{W}$ and $g\res{X} \in \End{X}$, which, by Remark \ref{rmk7}, imply that $g = g\res{W} \oplus g\res{X}$. Therefore, we have that
\[ g \in G_{(\bar{\Lambda})} \Leftrightarrow \begin{cases} g\res{W} \in \pi_W(G_{(\Lambda)}) \\ g\res{X}(X)=X  \end{cases}
\] 
and 
\[ g \in G_{(\bar{\Delta})} \Leftrightarrow \begin{cases} g\res{W} \in \pi_W(G_{(\Delta)}) \\ g\res{X}(X)=X.  \end{cases}
\] Hence $G_{(\bar{\Lambda})}= G_{(\bar{\Delta})}$ implies 
\begin{equation}\label{pi} \pi_W(G_{(\Lambda)})=\pi_W(G_{(\Delta)}).
\end{equation} 		
On the other hand 
\[ g \in G_{(\Lambda)} \Leftrightarrow g\res{W} \in \pi_W(G_{(\Lambda)})
\]		and
\[ g \in G_{(\Delta)} \Leftrightarrow g\res{W} \in \pi_W(G_{(\Delta)}).
\] Therefore, by \eqref{pi}, we get $G_{(\Lambda)}= G_{(\Delta)}$.
\end{enumerate}	

Finally, if $k < n-m$, then we can assume, by transitivity, that $W = \spanq{e_1, \dots,  e_k}$ and consider $V_0 = \spanq{e_1, \dots, e_{k+m}} \leq V$. Let $X= \spanq{e_{k+1}, \dots, e_{k+m}}$. Hence we can apply to $V_0$ the same argument as before, which leads to a contradiction.
	\end{proof}

The next lemma yields Lemma~\ref{l: psl} provided $G\leq \PGL_n(q)$.
	
\begin{lem}\label{l: pgl}
Let $G$ be an almost simple group with socle $\PSL_n(q)$, such that $\PSL_n(q) \unlhd G \leq \PGL_n(q)$, acting on $\Omega$, the set of all $m$-spaces in $V$. Then 
\[ \Height(G) <4 \log t.
\] 
Moreover,
\[\Height(G)<2\min(m,n-m)n.\]
\end{lem}
\begin{proof}
As usual, we may suppose $m\le n/2$ and $G=\mathrm{PGL}_n(q)$. Moreover, from the opening paragraph of Section~\ref{sec:G<=PGL(n,q)}, we may work with the linear group $\mathrm{GL}_n(q)$ and hence we may suppose $G=\mathrm{GL}_n(q)$.
Let $\Lambda$ be an independent set for the action of $G$ on $\Omega$, of maximal size.
The previous lemma allows us to order $\Lambda = \left\{W_1, \dots, W_k, \dots \right\}$ so that 
\[ \dim \spanq{W_1, \dots, W_j} > \dim \spanq{W_1, \dots, W_{j-1}}
\] for $j \leq k$, where $\dim {(\Lambda)}= \dim (\spanq{W_1, \dots, W_k}) > n-m$. Now $\Lambda$ is an independent set for $M=M_n(q)$ (see Section~\ref{preliminary} for the definition of independent set for the algebra of $n$-by-$n$ matrices over the finite field $\mathbb{F}_q$). Consider the following iterated process for $i \leq k$:
\begin{itemize}

 \item[\textbf{Step 1}] Let $W_1 \in \Omega$. We define $W = W_1$. We recall that we may assume, without loss of generality, that $W = \spanq{e_1, \dots, e_m}$ and that
\[ M_W = \left\{
\left(\begin{array}{@{}c|c@{}}
	\begin{matrix}
	A
	\end{matrix}
	& \begin{matrix}
	0 
	\end{matrix}\\
	\hline
\begin{matrix} B
\end{matrix} &
\begin{matrix}
	C
	\end{matrix}
	\end{array}\right) \mid A \in \M{m}, B \in \M{n-m,m}, C \in \M{n-m}  \right\}. \]

\item[\textbf{Step $i$}] We have $W_1, \dots, W_{i-1}$ and we define $W = \spanq{W_1, \dots, W_{i-1}}$. Let $\dim W = h$ and assume, without loss of generality, that $W=\spanq{e_1, \dots, e_{h}}$. Since we have chosen a specific ordering, we have that $W_i$ satisfies 
$$\dim \spanq{W_1, \dots, W_i} > \dim W.$$
Consider $V / W$ and $(W_i+W) / W\cong W_i/(W_i\cap W)$. Now let $\dim (W_i+W)/W= j$, observe that $j \in \left\{1, \dots, m\right\}$. Then we may assume
\[ W_i = \spanq{v_1, \dots, v_{m-j}, e_{h+1}, \dots, e_{h+j}},
\] for some $v_1, \dots, v_{m-j} \in W$ and $e_{h+1}, \dots, e_{h+j} \notin W$.

Hence $\dim \left(W \cap W_i\right) = m-j$ and we can write
\[ W_i = (W \cap W_i) \oplus \spanq{e_{h+1}, \dots, e_{h+j}}.
\]

Let $r\in \{h+1,\ldots,h+j\}$ and let $g \in M_{W_1, \dots, W_i}$; observe that $e_{r}^g \in W_i$. Hence $e_{r}^g = u_1 + u_2$, for some $u_1 \in W \cap W_i$ and $u_2 \in \spanq{e_{h+1}, \dots, e_{h+j}}$. Therefore, we have
\[ \dim \spanq{\right\{ e_r^g \mid g \in M_{W_1, \dots, W_i} \left\}}\le m,\,\,\, \mbox{ for all } r =h+1, \dots, h+j. 
\]
\end{itemize}

This iteration will end after $k$ steps and hence \[ \dim M_{W_1, \dots, W_{k}} \leq m(n-s) + ns,
\] where $s:=n-\dim \langle W_1,\ldots,W_k\rangle$. We conclude that $|\Lambda| \leq k+ m(n-s) + ns$ and so
we have
\[ \Height(G) \leq k+ m(n-s) + ns.
\]
Now observe that $k$ is maximum when $j=1$ at each step, in which case $k = n-m-s$. Hence
\begin{equation}\label{eqH0} 
\Height(G) \leq (n-m-s)+ ns+ m(n-s).
\end{equation}
By Lemma \ref{lemma6}, we have $s\le m-1.$ We conclude that the right-hand-side of $\eqref{eqH0}$ is maximized when $s=m-1$. Then
\[ \Height(G) \leq 2mn-m^2-m+1, 
\] and, as $m \geq 1$ implies $m^2+m-1 >0$, we conclude that
\begin{equation*} 
\Height(G) < 2mn.
\end{equation*}
Now, observe that
\[ \frac{mn}{m(n-m)} \leq 2 \Leftrightarrow m \leq \frac{n}{2}.
\] 
Then, as $\log q \geq 1$, we have
\[  \frac{mn}{m(n-m)} \leq 2 \leq 2 \log q.
\] This implies that
\[2mn \leq 4 \log q^{m(n-m)}
\]
Hence, using \eqref{eq: t bound}, we obtain 
\[\Height(G) < 2mn < 4 \log t.\qedhere
\]
\end{proof}

Note that, in the second author's thesis, a precise version of Lemma~\ref{l: pgl} is given for the case $m=1$. It turns out that, for $\PSL_n(q)\unlhd G\leq \PGL_n(q)$ acting on $\Omega$, the set of all $1$-spaces in $V$, we have
\[
 \Height(G,\Omega)=\begin{cases}
                    n, & \textrm{if $q=2$}; \\
                    2n-2, & \textrm{if $q>2$}.
                   \end{cases}
\]

\subsection{\texorpdfstring{$G\not\leq \PGL_n(q)$}{G not in PGL(n,q)}}

We can finally prove Lemma~\ref{l: psl}.

\begin{proof}[Proof of Lemma~\ref{l: psl}]
Let $q=p^f$. Suppose, first, that
\[ \PSL_n(q) \unlhd G \leq  \PGammaL_n(q) = \PGL_n(q) \rtimes \C_f.
\]
Now Lemma~\ref{l: chain} implies that
\[ \Height(G) \leq \Height(\PGL_n(q)) + \ell(\C_f),
\] 
where $\ell(\C_f)$ is the maximum length of a strictly descending chain of subgroups in $\C_f$. Since $\ell(\C_f) \leq \log f =\log \log_p q$,  we obtain 
\begin{equation*} 
\Height(G) < \Height(\PGL_n(q))+\log \log_p{q}.
\end{equation*}
From Lemma~\ref{l: pgl}, we immediately deduce~\eqref{eq:197} when $m<n/2$ and we also get
\begin{equation*} 
\Height(G) < 4\log t+\log \log_p{q}.
\end{equation*}
By \eqref{eq: t bound}, we easily conclude $\log \log_pq<\log t$ and hence
\[ \Height(G) < 5 \log t.
\]
We must deal with the possibility that $G\not\leq \PGammaL_n(q)$, i.e. $G$ contains a graph-automorphism. In this case $n\geq 4$, $m=\frac{n}{2}$ and $G$ contains a subgroup, $H$, of index $2$ that lies in $\PGammaL_n(q)$. By the argument above, we know that
\[
 \Height(H)<5\log t.
\]
Now Lemma~\ref{l: chain} implies that
\[
 \Height(G)<5\log t+1 <\frac{11}{2}\log t. \qedhere
\]
\end{proof}

We remark that in Lemma~\ref{l: pgl} we chose a particular ordering of our independent set $\Lambda$ when we came to study the associated stabilizer chain. We can do this because we are studying $\Height(G)$ rather than $\Irred(G)$ -- it is precisely this ordering step which has prevented us from proving a strong enough version of Theorem~\ref{t: irred} to confirm Conjecture~\ref{conj: I}.

\subsection{\texorpdfstring{$G_0=\PSL_n(q)$}{G0=PSL(n,q)} and elements of \texorpdfstring{$\Omega$}{Omega} are pairs of subspaces}

In this section we will consider the primitive subspace actions of those almost simple groups with socle $\mathrm{PSL}_n(q)$ that contain an automorphism of the Dynkin diagram. 

Let $n \geq 3$, and fix $m$, an integer satisfying $1\le m<\frac{n}{2}$.  We denote by $\Omega_m$ the set of all $m$-subspaces of $V$. 
We consider the action of $G$ on $\Omega^{(i)}$, for $i=1,2$, where

\begin{align*} &\Omega^{(1)}= \left\{ \{U,W\} \mid  U,W \leq V, \dim{U}=m, \dim{W}=n-m \text{ with } m< \frac{n}{2} \text{ and } U \oplus W= V \right\},\\
&\Omega^{(2)}= \left\{ \{U,W\} \mid  U,W \leq V, \dim{U}=m, \dim{W}=n-m \text{ with } m< \frac{n}{2} \text{ and } U<W \right\}.
\end{align*} 

The main result of this section is the following.

\begin{lem}\label{l: psl 2} 
Let $G$ be an almost simple group with socle $\PSL_n(q)$ acting on $\Omega^{(1)}$ or $\Omega^{(2)}$. Then
\[ \Height(G) < \frac{17}{2}\log t.
\]
\end{lem}

We set $t^{(i)} := |\Omega^{(i)}|$, for $i=1,2$. Then we have
\[t^{(1)}=\frac{q^{m(n-m)}\prod_{i=1}^{n}{(q^i-1)}}{\prod_{i=1}^{m}{(q^i-1)}\prod_{i=1}^{n-m}{(q^i-1)}} \,\,\textrm{    and   }\,\, t^{(2)}=\frac{\prod_{i=1}^{n}{(q^i-1)}}{\prod_{i=1}^{m}{(q^i-1)^2}\prod_{i=1}^{n-2m}{(q^i-1)}}.\]

We saw already, in \eqref{eq: t bound}, that $|\Omega_m|>q^{m(n-m)}$. We obtain immediately that
\begin{equation}\label{eq: t bound 2}
 t^{(i)}\geq q^{m(n-m)}.
\end{equation}

For the time being, for various technical reasons, we consider $G\le \PGammaL_n(q)$ (although the corresponding actions are not primitive any longer). We simultaneously study the action of $G$ on $\Omega^{(1)}$ and $\Omega^{(2)}$. Let $\Lambda \subset \Omega^{(i)}$ be an independent set for the action of $G$ on $\Omega^{(i)}$, with $k:=|\Lambda|$. We define 
\begin{align} 
 &\mathcal{U} = \lbrace  U\in\Omega_m\mid \textrm{there exists }W\in \Omega_{n-m} \textrm{ with }\{U,W\}\in\Lambda \rbrace; \\
 & \mathcal{W} =\lbrace  W\in\Omega_{n-m}\mid \textrm{there exists }U\in \Omega_{m} \textrm{ with }\{U,W\}\in\Lambda \rbrace.
\end{align}
Let $k_U:=|\mathcal{U}|$ and $k_W:=|\mathcal{W}|$. Clearly, $k_U,k_W \leq |\Lambda|=k$.

Notice that, since $G\le\PGammaL_n(q)$, $G$ acts naturally on $\Omega_m$ and $\Omega_{n-m}$ however observe that, \textit{a priori}, the independence of $\Lambda$ does not imply the independence of $\Lambda_U$ and $\Lambda_W$. The result that we need is the following.

\begin{lem}\label{l: k}
$k \leq \Height(G,\Omega_m)+ \Height(G,\Omega_{n-m}).$
\end{lem}

In order to prove Lemma~\ref{l: k} we first need to define a graph $\Gamma$ associated with $\Lambda$ in the following way: $\Gamma= (V,\Lambda)$ where $V= \mathcal{U} \cup \mathcal{W}$ is the vertex set and the edge set is $\Lambda$. 

From here, we will first prove a general fact about the length of paths in the graph $\Gamma$, and we will then split our proof of Lemma~\ref{l: k} into two steps: the first is a special case which illustrates the general case; this general case will come second.

\begin{lem}\label{lem:19}The graph $\Gamma$ can only have paths of length at most 2.
\end{lem}
\begin{proof}
Assume that in $\Gamma$ we have a path of length at least 3. Then, relabelling if necessary, we have
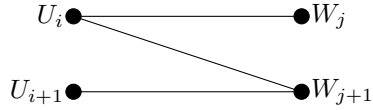
\begin{center}

\label{fig:4.2} \begin{tikzpicture} 
\draw (0,0) --(3,0);
\draw (0,0) --(3,-1);
\draw (0,-1)--(3,-1);
\draw [fill] (0,0) circle [radius=0.1];
\node [left] at (0,0) {$U_i$};
\draw [fill] (0,-1) circle [radius=0.1];
\node [left] at (0,-1) {$U_{i+1}$};
\draw [fill] (3,0) circle [radius=0.1];
\node [right] at (3,0) {$W_j$};
\draw [fill] (3,-1) circle [radius=0.1];
\node [right] at (3,-1) {$W_{j+1}$};
\end{tikzpicture}
\end{center} 
for some $1 \leq i < k_U$, $1 \leq j < k_W$.
As $G\le \PGammaL_n(q)$, $G$ preserves the two parts $\Lambda_U$ and $\Lambda_W$ of the natural bipartition of $\Gamma$. Hence
$$G_{\{U_i,W_j\},\{U_{i+1},W_{j+1}\}}=G_{U_i}\cap G_{W_j}\cap G_{U_{i+1}}\cap G_{W_{j+1}}=G_{\{U_i,W_j\},\{U_{i},W_{j+1}\},\{U_{i+1},W_{j+1}\}}.$$ As $\{U_i,W_{j}\}$, $\{U_{i},W_{j+1}\}$, $\{U_{i+1},W_{j+1}\}$ are elements in $\Lambda$,  this contradicts the independence of $\Lambda$.
\end{proof}

\subsection{Suppose that \texorpdfstring{$\Gamma$}{Gamma} is a complete matching}
In other words we suppose $\Gamma$ has no paths of length $2$. This case will illustrate the general argument very well.

In this case $\Lambda=\{\{U_1,W_1\},\ldots,\{U_k,W_k\}\}$. Let us consider the action of $G$ on $\Omega_m$. By Lemma \ref{l: subset}, there exists $\mathcal{I} \subset \{1, \dots, k\}$, say $\mathcal{I}= \{i_1, \dots, i_s\}$, such that $\{U_{i_1}, \dots, U_{i_s}\}$ is an independent set and 
$$ \bigcap_{j=1}^{k}{G_{U_i}}=\bigcap_{j=1}^{s}{G_{U_{i_j}}}. $$
After relabelling elements, we can assume $\mathcal{I}= \{1, \dots, s\}$. 
This yields $$s \leq \Height(G, \Omega_{m}).$$ 
Let $$H := \bigcap_{i=1}^{s}{\left(G_{U_i} \cap G_{W_i}\right)}$$ and consider its action on $\Omega^{(i)}$, for $i \in \{1,2\}$. Recall that, by Lemma \ref{l: subgroup}, the set $$\Delta=\{\{U_{s+1}, W_{s+1}\}, \dots, \{U_k,W_k\}\}$$ is independent with respect to $H$. Now consider the action of $H$ on $\Omega_{n-m}$.

\begin{lem} The set $\{W_{s+1}, \dots, W_k\}$ is independent with respect to the action of $H$ on $\Omega_{n-m}$.
\end{lem}
\begin{proof} Assume, for contradiction, that the set $\{W_{s+1}, \dots, W_k\}$ is not independent with respect to $H$. Then, by Lemma~\ref{l: stab}, there would exist $j \in \{s+1, \dots k\}$ such that

$$\bigcap_{i=s+1}^{k}{H_{W_i}}=\bigcap_{\substack{i=s+1\\i \neq j}}^{k}{H_{W_i}}. $$

Observe that
\[ \bigcap_{i=1}^{s}{G_{U_i}}=\bigcap_{i=1}^{k}{G_{U_i}}=\bigcap_{\substack{i=1\\i \neq j}}^{k}{G_{U_i}}.
\] Therefore we have
\begin{align*} 
G_{(\Lambda)}&= H \cap \bigcap_{i=s+1}^{k}{G_{W_i}} = \bigcap_{i=s+1}^{k}{H_{W_i}} = \bigcap_{\substack{i=s+1 \\ i \neq j}}^{k}{H_{W_i}} \\
&= H \cap \left(\bigcap_{\substack{i=s+1 \\ i \neq j}}^{k}{G_{W_i}}\right) = \left( \bigcap_{i=1}^{s}{G_{U_i}} \cap \bigcap_{i=1}^{s}{G_{W_i}} \right) \cap \left(\bigcap_{\substack{i=s+1 \\ i \neq j}}^{k}{G_{W_i}} \right) \\
&= \left( \bigcap_{\substack{i=1 \\ i \neq j}}^{k}{G_{U_i}}  \right) \cap \left( \bigcap_{\substack{i=1 \\ i \neq j}}^{k}{G_{W_i}} \right) = \bigcap_{\substack{i=1 \\ i \neq j}}^{k}{\left(G_{U_i} \cap G_{W_i}\right)} \\
&= G_{(\Lambda \setminus \{\{U_j,W_j\}\})}.
\end{align*}
This contradicts the independence of $\Lambda$.
\end{proof}
The previous lemma implies that
\[k-s \leq \Height(H, \Omega_{n-m}).
\]
Now, by Lemma~\ref{l: subgroup}, we have $\Height(H, \Omega_{n-m}) \leq \Height(G, \Omega_{n-m})$.
Putting these things together yields that 
\[k \leq \Height(G,\Omega_m)+ \Height(G,\Omega_{n-m}),\]
and Lemma~\ref{l: k} is proved in this special case.

\subsection{Suppose that \texorpdfstring{$\Gamma$}{Gamma} is not a complete matching}
The general argument is very similar but requires some more notation. After reordering the vertices of $\Gamma$ and using a suitable labelling, the fact that $\Gamma$ has no paths of length $3$ implies that $\Gamma$ is isomorphic to the graph in Figure \ref{fig:4.4}, where $\ell_1,\dots,\ell_b, s_1,\ldots, s_c \geq 2$.

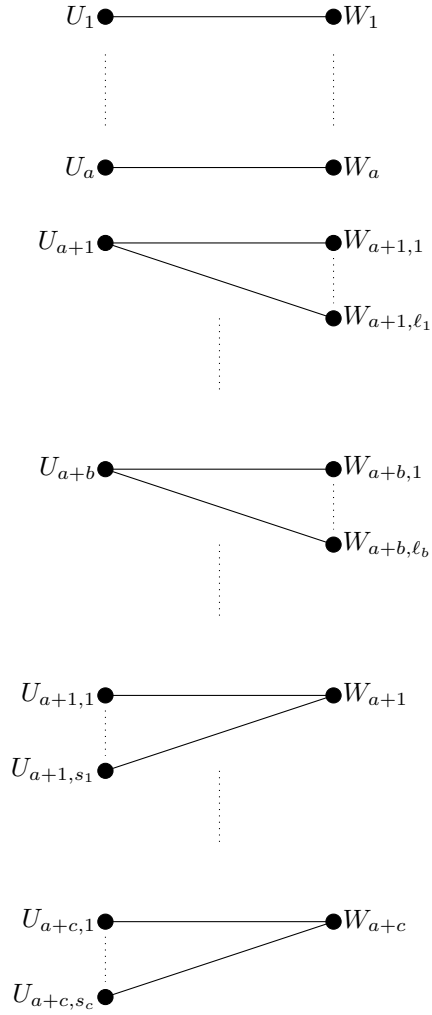
\begin{figure}
\centering

\begin{tikzpicture}
\draw (0,0) --(3,0);
\draw (0,-2) --(3,-2);
\draw[dotted](0,-.5)--(0,-1.5);
\draw[dotted](3,-.5)--(3,-1.5);
\draw (0,-3) --(3,-3);
\draw (0,-3) --(3,-4);
\draw[dotted](3,-3.2)--(3,-3.8);
\draw[dotted](1.5,-4)--(1.5,-5);
\draw (0,-6) --(3,-6);
\draw (0,-6) --(3,-7);
\draw[dotted](3,-6.2)--(3,-6.8);
\draw[dotted](1.5,-7)--(1.5,-8);
\draw[dotted](1.5,-10)--(1.5,-11);
\draw (0,-9) --(3,-9);
\draw (0,-10) --(3,-9);
\draw [dotted](0,-9.2)--(0,-9.8);
\draw (0,-12) --(3,-12);
\draw (0,-13) --(3,-12);
\draw [dotted](0,-12.2)--(0,-12.8);

\draw [fill] (0,0) circle [radius=0.1];
\node [left] at (0,0) {$U_1$};
\draw [fill] (0,-2) circle [radius=0.1];
\node [left] at (0,-2) {$U_a$};
\draw [fill] (0,-3) circle [radius=0.1];
\node [left] at (0,-3) {$U_{a+1}$};
\draw [fill] (0,-6) circle [radius=0.1];
\node [left] at (0,-6) {$U_{a+b}$};
\draw [fill] (0,-9) circle [radius=0.1];
\node [left] at (0,-9) {$U_{a+1,1}$};
\draw [fill] (0,-10) circle [radius=0.1];
\node [left] at (0,-10) {$U_{a+1,s_1}$};
\draw [fill] (0,-12) circle [radius=0.1];
\node [left] at (0,-12) {$U_{a+c,1}$};
\draw [fill] (0,-13) circle [radius=0.1];
\node [left] at (0,-13) {$U_{a+c,s_c}$};

\draw [fill] (3,0) circle [radius=0.1];
\node [right] at (3,0) {$W_1$};
\draw [fill] (3,-2) circle [radius=0.1];
\node [right] at (3,-2) {$W_a$};
\draw [fill] (3,-3) circle [radius=0.1];
\node [right] at (3,-3) {$W_{a+1,1}$};
\draw [fill] (3,-4) circle [radius=0.1];
\node [right] at (3,-4) {$W_{a+1,\ell_1}$};
\draw [fill] (3,-6) circle [radius=0.1];
\node [right] at (3,-6) {$W_{a+b,1}$};
\draw [fill] (3,-7) circle [radius=0.1];
\node [right] at (3,-7) {$W_{a+b,\ell_b}$};
\draw [fill] (3,-9) circle [radius=0.1];
\node [right] at (3,-9) {$W_{a+1}$};
\draw [fill] (3,-12) circle [radius=0.1];
\node [right] at (3,-12) {$W_{a+c}$};

\end{tikzpicture}
\caption{The graph $\Gamma$ in the general case.}
\label{fig:4.4}
\end{figure}

 Observe that, by definition, we have $k_U$ vertices on the left, $k_W$ vertices on the right and $k$ edges, because these are the elements of $\Lambda$. 
Let us consider the action of $G$ on $\Omega_m$, that is we focus on the action of $G$ on the vertices on the left-hand side of the graph. By Lemma \ref{l: subset}, there exists a subset $\mathcal{U}_1$ of $\mathcal{U}$, such that $\mathcal{U}_1$ is an independent set and 
\begin{equation}\label{eq1230} \bigcap_{U \in \mathcal{U}}{G_{U}}= \bigcap_{U \in \mathcal{U}_1}{G_{U}}.
\end{equation} Let $s := |\mathcal{U}_1|.$ We reorder the vertices on the left-hand side of the graph so that the elements of $\mathcal{U}_1$ occur in the top $s$ positions. Simultaneously, we take to the top exactly one edge, chosen arbitrarily, having one end in $U$, for each $U \in \mathcal{U}_1$. Note that here we allow crossings between edges. Observe that Lemma \ref{lem:19} implies that this operation is well-defined. We have
\[ s \leq \mathrm{H}(G,\Omega_m).
\]
Let $\Lambda_{\mathcal{U}_1} \subset \Lambda$ be the set of those chosen edges and note that $|\Lambda_{\mathcal{U}_1}|=s$. We define
\begin{align*} \mathcal{W}_1 &= \left\{ W \in \mathcal{W} \mid W \text{ is an end-point for an edge in } \Lambda_{\mathcal{U}_1}  \right\} \\
\mathcal{W}_2 &=\mathcal{W}\setminus \mathcal{W}_1.
\end{align*}
Observe that $|\mathcal{W}_1| \leq s$. 

\begin{lem}Let $H := G_{(\Lambda_{\mathcal{U}_1})}$. Let $\lambda \in \Lambda \setminus \Lambda_{\mathcal{U}_1}$ and let $U_{\lambda}$ and $W_{\lambda}$ be, respectively, its left-hand end-point and right-hand end-point. Then we have that
\begin{enumerate}
\item the function
\begin{align*} f \colon &\Lambda \setminus \Lambda_{\mathcal{U}_1} \to \mathcal{W}_2 \\
&\lambda \mapsto W_{\lambda}
\end{align*}
is a bijection. In particular, $|\mathcal{W}_2|=|\Lambda \setminus \Lambda_{\mathcal{U}_1}| =k-s$;
\item $G_{(\Lambda)}= H_{(\mathcal{W}_2)}$.
\end{enumerate}
\end{lem}
\begin{proof}
To prove claim (1), we first prove that $f$ is well-defined.  We claim that $W_{\lambda} \in \mathcal{W}_2$. Indeed, if $W_{\lambda} \in \mathcal{W}_1$, then
\begin{align*} 
G_{(\Lambda_{\mathcal{U}_1})} &= \bigcap_{\substack{U \in \mathcal{U}_1 \\ W \in \mathcal{W}_1}}{\left(G_{U} \cap G_{W}\right)} \\
&= \left(\bigcap_{U \in \mathcal{U}_1 \cup \left\{U_{\lambda}\right\}}{G_U}\right)\cap\left(\bigcap_{W \in \mathcal{W}_1}{G_W}\right) \\
&= G_{(\Lambda_{\mathcal{U}_1} \cup \left\{\lambda\right\})},
\end{align*} 
which contradicts the independence of $\Lambda$. This fact implies that $f$ is well-defined. Moreover, it is onto by definition. Finally, we prove that $f$ is one-to-one. Assume that there exist two distinct edges $\lambda, \mu \in \Lambda \setminus \Lambda_{\mathcal{U}_1}$ such that $W_{\lambda}=W_{\mu}$. Then, by using \eqref{eq1230}, we have
\begin{align*} 
G_{(\Lambda_{\mathcal{U}_1} \cup \left\{\lambda,\mu\right\})} &= \bigcap_{\substack{U \in \mathcal{U}_1 \\ W \in \mathcal{W}_1}}{\left(G_{U} \cap G_{W}\right)} \cap \left(G_{U_{\lambda}}\cap G_{W_{\lambda}}\right) \cap \left(G_{U_{\mu}}\cap G_{W_{\mu}}\right)\\
&= \left(\bigcap_{U \in \mathcal{U}_1 \cup \left\{U_{\lambda}\right\} \cup \left\{U_{\mu}\right\}}{G_U}\right)\cap \left(\bigcap_{W \in \mathcal{W}_1 \cup \left\{W_{\lambda}\right\}}{G_W}\right) \\
&= \left(\bigcap_{U \in \mathcal{U}_1 \cup \left\{U_{\lambda}\right\}}{G_U}\right)\cap \left(\bigcap_{W \in \mathcal{W}_1 \cup \left\{W_{\lambda}\right\}}{G_W}\right) \\
&= G_{(\Lambda_{\mathcal{U}_1} \cup \left\{\lambda\right\})},
\end{align*} 
which is a contradiction. Hence $f$ is a bijection, in particular $|\mathcal{W}_2|=|\Lambda \setminus \Lambda_{\mathcal{U}_1}|= k-s$. 

For part (2), by using the previous result and \eqref{eq1230}, we have
\begin{align*} 
G_{(\Lambda)} &= G_{(\Lambda_{\mathcal{U}_1})} \cap G_{(\Lambda \setminus \Lambda_{\mathcal{U}_1} )} = H \cap G_{(\Lambda \setminus \Lambda_{\mathcal{U}_1} )} \\
&= H \cap \left( \bigcap_{\lambda \in \Lambda \setminus \Lambda_{\mathcal{U}_1}}{G_{U_{\lambda}} \cap G_{W_{\lambda}}}\right) \\
&= \left(H \cap \bigcap_{\lambda \in \Lambda \setminus \Lambda_{\mathcal{U}_1}}{G_{U_{\lambda}}}\right) \cap \left(H \cap \bigcap_{\lambda \in \Lambda \setminus \Lambda_{\mathcal{U}_1}}{G_{W_{\lambda}}}\right)  \\
&= H \cap \left(\bigcap_{W \in \mathcal{W}_2}{G_W}\right) =H_{(\mathcal{W}_2)}.\qedhere
\end{align*}

\end{proof}

Now consider the action of $H$ on $\Omega_{n-m}$, namely we focus on the vertices on the right-hand side of the graph. As a consequence of the previous lemma, we obtain the following corollary.

\begin{cor} The set $\mathcal{W}_2$ is independent for the action of $H$ on $\Omega_{n-m}$. 
\end{cor}

\begin{proof} Assume not, then there would exist $\lambda \in \Lambda \setminus \Lambda_{\mathcal{U}_1}$ such that
\[ H_{(\mathcal{W}_2)} =\bigcap_{W \in \mathcal{W}_2 }{H_{W}} = \bigcap_{\substack{W \in \mathcal{W}_2 \\ W \neq W_{\lambda}} }{H_{W}}.
\] 
If $U_\lambda \in \mathcal{U}_1$, then there exists $\mu\in \Lambda_{\mathcal{U}_1}$ such that $U_\lambda=U_\mu$. If $U_\lambda \not\in \mathcal{U}_1$, then $\bigcap\limits_{U\in \mathcal{U}} G_U = \bigcap\limits_{U\in \mathcal{U}, U\neq U_\lambda} G_U$.

In either case, this, along with \eqref{eq1230} and part (2) of the previous lemma, implies that
\begin{align*}
 G_{(\Lambda \setminus \left\{\lambda \right\})}&= \left(\bigcap_{U \in \mathcal{U} }{G_{U}}\right) \cap \left(\bigcap_{\substack{W \in \mathcal{W} \\ W \neq W_{\lambda}}}{G_{W}}\right)\\
&= \left(\bigcap_{U \in \mathcal{U}_1}{G_{U}} \cap \bigcap_{W \in \mathcal{W}_1}{G_{W}}\right) \cap \left(\bigcap_{\substack{W \in \mathcal{W}_2 \\ W \neq W_{\lambda}}}{G_{W}} \right) \\
&= H \cap \bigcap_{\substack{W \in \mathcal{W}_2 \\ W \neq W_{\lambda}} }{G_{W}} \\
&= \bigcap_{\substack{W \in \mathcal{W}_2 \\ W \neq W_{\lambda}} }{H_{W}}=H_{(\mathcal{W}_2)}=G_{(\Lambda)}.
\end{align*}

This contradicts the independence of $\Lambda$.
\end{proof}

The previous lemma implies that
\[k-s \leq \Height(H, \Omega_{n-m}).
\]
Now, by Lemma~\ref{l: subgroup}, we have $\Height(H, \Omega_{n-m}) \leq \Height(G, \Omega_{n-m})$.
Putting these things together yields that 
\[k \leq \Height(G,\Omega_m)+ \Height(G,\Omega_{n-m}),\]
and Lemma~\ref{l: k} is proved in the general case. We are ready to prove Lemma~\ref{l: psl 2}.

\begin{proof}[Proof of Lemma~\ref{l: psl 2}]
 Suppose, first that $G\leq \PGammaL_n(q)$. Recall that the actions of $G$ on $\Omega_m$ and $\Omega_{n-m}$ are permutation isomorphic. Thus, Lemma~\ref{l: k} and \eqref{eq:197} imply that
\[ \Height(G, \Omega^{(i)}) < 4mn+2\log\log_pq.
\]
(Here we are tacitly assuming that $m<n/2$ because, for instance,~\eqref{eq:197} does use the hypothesis that $m<n/2$. However, this is not a restriction because the duality conjugates the $m$-space stabilisers in $(n-m)$-space stabilisers.)

Now let $G \nleq \PGammaL_n{(q)}$, that is $G$ contains the inverse transpose automorphism. Let $H = G \cap \PGammaL_n(q)$. Then $H$ is a normal subgroup of $G$ of index 2. Then, by Lemma \ref{l: chain} we have 
\[ \Height(G, \Omega^{(i)}) \leq \Height(H, \Omega^{(i)})+1,
\]
and therefore
\begin{equation}\label{eq:4.11}
\Height(G, \Omega^{(i)}) < 4mn+2\log\log_pq +1.
\end{equation}
In view of~\eqref{eq: t bound 2},  to prove the result it is sufficient to check that
\[
 4mn+2\log\log_pq +1  < \frac{17}{2}\log(q^{m(n-m)}). 
\] 
The result follows directly.
\end{proof}

\subsection{The other classical groups}

In this subsection we suppose that $G_0={\rm Cl}(V)$ is one of the other classical groups defined on a vector space $V$ of dimension $n$ over $\Fq$. In this case, $m$ is an integer such that $0<m<n$ and $\Omega$ is a $G$-orbit on the set of totally isotropic/ totally singular/ non-degenerate subspaces of dimension $m$ in $V$. As usual we set $t=|\Omega|$.

The result that we need is the following.

\begin{lem}\label{l: other classical} 
Let $G$ be an almost simple group with socle ${\rm Cl}_n(q)$, and consider $G$ as a permutation group in a subspace action. Then
\[ \Height(G) < \frac{17}{2}\log t.
\]
\end{lem}


In \cite[Table~4.1.2]{burness_giudici} we find a list of the degrees of all such actions and, using the notation just established, one can easily verify the following fact.

\begin{lem}\label{l: t lower}
Either $G_0\cong \POmega_n^+(q)$ and $m=\frac{n}{2}$ or $t>q^{\frac12 m(n-m)}$.
\end{lem}
\begin{proof}
The proof of this lemma is rather tedious: it is simply a case-by-case analysis of the cases described in Table~4.1.2 of~\cite{burness_giudici}. 

We make a couple of brief remarks about the proof and leave the details to the reader. Observe, first, that for a number of the cases (III, X with $q$ odd, XI with $q$ odd, XIII, XV, XVI), $t$ is divisible by $q^{\frac12 m(n-m)}$ and so the result follows immediately.

For the rest, one must demonstrate the bound directly. To illustrate we consider Case VI which is where the exception occurs; in this case $G_0\cong \POmega_n^+(q)$ and $m\leq \frac{n}{2}$. Then we have
\begin{align*}
 t &= \frac{(q^{n/2}-1)(q^{(n-2m)/2}+1) \prod_{i=(n-2m)/2+1}^{(n-2m)/2+(m-1)}(q^{2i}-1)}{\prod_{i=1}^m(q^i-1)} \\
 &= \frac{(q^{n/2}-1)(q^{(n-2m)/2}+1)}{q^m-1} \cdot \prod_{i=1}^{m-1}\left(\frac{q^{n-2m+2i}-1}{q^i-1}\right) \\
 &> \frac{(q^{n/2}-1)(q^{(n-2m)/2}+1)}{q^m-1} \cdot \prod_{i=1}^{m-1}q^{n-2m+i} \\
 &> q^{n/2-m+n/2-m} \cdot q^{(n-2m)(m-1)+\frac12(m-1)m} \\
 &=q^{nm-\frac32m^2-m}.
\end{align*}
Now it is easy to check that the bound is satisfied if and only if $m<\frac{n}{2}$, as required.
\end{proof}

We are ready to prove Lemma~\ref{l: other classical}.

\begin{proof}[Proof of Lemma~\ref{l: other classical}]
 Observe first that $G\leq \PGammaL_n(q)$ and that the action we are studying is an action on $\Gamma$, the set of $m$-spaces. Then Lemma~\ref{l: subgroup} implies that
\[
 \Height(G,\Omega)\leq \Height(\PGammaL_n(q), \Gamma).
 \]
In Lemma~\ref{l: psl} we give an upper bound for $\Height(\PGammaL_n(q), \Gamma)$:
 \[
\Height(G) <  2\min(m,n-m)n + \log \log_p{q}.
\]
Thus, in light of Lemma~\ref{l: t lower}, provided we do not have $(G_0,m)=(\POmega_n^+(q),\frac{n}{2})$ it is enough to prove that
\[
 2\min(m,n-m)n + \log \log_p{q} < \frac{17}{2}\log (q^{\frac12 m(n-m)}).
\]
This is easily verified. Assume that $(G_0,m)=(\POmega_n^+(q),\frac{n}{2})$. In this case $t\geq q^{\frac{n^2}{8}-\frac{n}{4}}$. Lemma~\ref{l: omega 8} implies that we can assume that $n\geq 10$ and thus
\[
 |G|\leq 2fq^{n(n-2)/4}(q^{n/2}-1) \prod_{i=1}^{n/2-1}(q^{2i}-1) < 2f q^{\frac{n^2}{2}-\frac{n}{2}} \leq q^{\frac{n^2}{2}-\frac{n}{2}+1}.
\]
Now observe that
\[
 \Height(G) \leq \log |G| < \log(q^{\frac{n^2}{2}-\frac{n}{2}+1}) < 5 \log( q^{\frac{n^2}{8}-\frac{n}{4}}) < 5 \log t,
\]
and we are done.
\end{proof}

Proposition~\ref{p: as height} is a consequence of Lemmas~\ref{l: psl}, \ref{l: psl 2} and \ref{l: other classical}. 

\begin{proof}[Proof of Theorem~\ref{t: height}]
Observe that $17/2+3/(2x)<9$, as long as $x>3$ (so $|\Delta|>8$). With this remark, the proof follows from Theorem~\ref{t: irred}, and Propositions~\ref{p: I CD and PA} and \ref{p: as height}, except when $G$ is a primitive group of PA type with $G\le H\wr \Sym(\ell)$ and the domain of $G$ is $\Omega=\Delta^\ell$ and $|\Delta|\le 8$. If $H=\Alt(\Delta)$ or $H=\Sym(\Delta)$, then $G$ is a large-base group and hence we may omit these cases for the rest of the proof. As $H$ is a primitive group of degree $5,6,7$ or $8$ and as $H$ is neither $\Alt(\Delta)$ nor $\Sym(\Delta)$, we deduce that one of the following holds
\begin{itemize}
\item $|\Delta|=6$ and $H=\PSL_2(5)$ or $H=\PGL_2(5)$ in its natural action on the projective line,
\item $|\Delta|=7$ and $H=\PSL_3(2)$ in its natural action on the projective plane,
\item $|\Delta|=8$ and $H=\PSL_2(7)$ or $H=\PGL_2(7)$ in its natural action on the projective line.
\end{itemize}
In each of these groups, we have $\mathrm{H}(H)\le 3$ and hence from~\eqref{e: chain} we have $\mathrm{H}(G)\le 3\ell+3\ell/2=9\ell/2$. As $9\ell/2<9\log(t)$, the proof is completed.
\end{proof}

\end{document}